\documentclass[11pt]{article}

\usepackage[dvipdfmx]{graphicx}
\usepackage{amssymb,amsmath}
\usepackage{fullpage}
\usepackage{amsthm}
\usepackage{thmtools}
\usepackage{enumerate}
\usepackage{natbib}

\usepackage{algorithm,algpseudocode}
\usepackage{caption}
\usepackage[colorlinks=true,allcolors=blue]{hyperref}

\def\ep{\varepsilon}
\newcommand{\dom}{\operatorname{dom}}
\newcommand{\sign}{\operatorname{sign}}
\newcommand{\Argmin}{\operatornamewithlimits{Argmin}}
\newcommand{\argmin}{\operatornamewithlimits{argmin}}

\def\E{\mathbb{E}}
\newcommand{\norm}[1]{\left\|#1\right\|}
\newcommand{\innprod}[2]{\left\langle#1,#2\right\rangle}

\def\RR{\mathbb{R}}
\def\cB{{\mathcal B}}

\def\ZZ{\mathbb{Z}}

\def\cT{\mathcal{T}}

\declaretheorem{theorem}

\declaretheorem{corollary}
\declaretheorem{lemma}

\declaretheorem[style=definition,qed=\qedsymbol]{example}
\declaretheorem[style=remark,qed=\qedsymbol]{remark}
\declaretheorem[style=definition,qed=\qedsymbol]{assumption}

\begin{document}

\title{A Parameter-Free Conditional Gradient Method for Composite Minimization under H\"older Condition}

\author{
Masaru Ito
\thanks{
Department of Mathematics, College of Science and Technology, Nihon University, Japan (Email: {\tt ito.masaru@nihon-u.ac.jp}).}
\and
Zhaosong Lu
\thanks{
Department of Industrial and Systems Engineering, University of Minnesota, USA (Email: {\tt zhaosong@umn.edu}, {\tt he000233@umn.edu}).}
\and
Chuan He 
\footnotemark[2]
}
\date{May 28, 2023}

\maketitle

\begin{abstract}
In this paper we consider a composite optimization problem that minimizes the sum of a weakly smooth function and a convex function with either a bounded domain or a uniformly convex structure. In particular, we first present a parameter-dependent conditional gradient method for this problem, whose step sizes require prior knowledge of the parameters associated with 
the H\"older continuity of the gradient of the weakly smooth function, and establish its rate of convergence. Given that these parameters could be unknown or known but possibly conservative,  such a method may suffer from implementation issue or slow convergence. 
We therefore propose a parameter-free conditional gradient method whose step size is determined by using a constructive local quadratic upper approximation and an adaptive line search scheme, without using any problem parameter.
We show that this method achieves the same rate of convergence as the parameter-dependent conditional gradient method. 
Preliminary experiments are also conducted and illustrate the superior performance of the parameter-free conditional gradient method over the methods with some other step size rules.
\\

\\
\end{abstract}


\section{Introduction}

In this paper we consider a composite optimization problem in the form of
\begin{equation}\label{main-prob}
\varphi^*=\min_{x \in \E} \left\{\varphi(x):= f(x)+g(x)\right\},
\end{equation}
where $\E$ is a finite dimensional real Hilbert space endowed with an inner product 
$\innprod{\cdot}{\cdot}$, and the functions $f,g:\E\to \RR\cup\{+\infty\}$ are proper and 
lower-semicontinuous. Assume that $\varphi^*$ is finite and attainable, and that $g$ is a convex function, while $f$ is possibly nonconvex.

In the recent years conditional gradient methods \cite[e.g.,][]{Bach15,HJN15,Nes18,Ghadimi19} have been developed for solving problem \eqref{main-prob}, which generate the iterates $\{x_t\}$ by the following scheme:
\begin{equation}\label{cond-grad-scheme}
v_t \in \Argmin_{x\in \E}\{\innprod{\nabla{f}(x_t)}{x}+g(x)\},\quad x_{t+1}=(1-\tau_t)x_t+\tau_t v_t
\end{equation}
for each $t\geq 0$, where $\tau_t \in [0,1]$ is a step size chosen by a certain rule. 
These methods originate from the Frank-Wolfe method \cite[]{FW} that was initially proposed for quadratic programming and later studied for more general or structured problems  \cite[e.g.,][]{LevPol,DemRub,Dunn79,BT04}.
Conditional gradient methods have found rich applications in machine learning and statistics, where $f$ and $g$ are typically a loss function and a regularizer, respectively. The advantage of conditional gradient methods in these applications is that the solution $v_t$ to the subproblem in \eqref{cond-grad-scheme} is efficiently computable and also has some desirable property, such as preserving sparsity or low rank  \cite[]{Hazan08,Cla10,Jaggi11,HJN15,FG16,Nes18}.

The conditional gradient methods provide a computable quantity $\delta_t$, commonly referred to as the Frank-Wolfe gap,  which is given by
\begin{equation}\label{def-FW-gap}
\delta_t := \innprod{\nabla{f}(x_t)}{x_t-v_t} + g(x_t)-g(v_t) = \max_{x \in \E}\{\innprod{\nabla f(x_t)}{x_t-x}+g(x_t)-g(x)\} \geq 0.
\end{equation}
Notice that $\delta_t =0$ if and only if $x_t$ is a stationary point of $\varphi$. In addition, if $f$ is convex, one can have $\varphi(x_t)-\varphi^* \le \delta_t$ (see Lemma~\ref{lem:FW-gap-bound-convex}). Consequently, $\delta_t \leq \ep$ is often used as a termination criterion for the conditional gradient methods.

The choice of step sizes is crucial for the performance of the conditional gradient methods, which is typically measured by the iteration complexity, namely,  the worst-case number of iterations for reaching $\delta_t \le \ep$ or $\varphi(x_t)-\varphi^* \le \ep$ for a prescribed tolerance $\ep>0$. There are numerous studies \cite[e.g.,][]{FW,LevPol,Dunn79,FG16} on the step size rule when $f$ is $L$-smooth, i.e., $\nabla f$ is Lipschitz continuous with constant $L$ under a norm $\norm{\cdot}$. For example, the choice 
\begin{equation}\label{eq:short-step size}
\tau_t = \min\left\{1, \frac{\delta_t}{L\norm{x_t-v_t}^2}\right\}
\end{equation}
guarantees an iteration complexity of $O(\ep^{-2})$ for reaching $\delta_t \le \ep$ \cite[]{Lac16}. When $f$ is additionally convex, it can be improved to $O(\ep^{-1})$, which can also be achieved by the step size $\tau_t = \frac{2}{t+2}$, a well-known choice for convex $f$ \cite[]{Jaggi11,FG16}.

More generally, step size rules were studied when $f$ is weakly smooth, that is, $\nabla f$ is H\"older continuous with exponent $\nu \in (0,1]$. In particular, when $f$ is additionally convex, the choice $\tau_t=\frac{2}{t+2}$ ensures an iteration complexity of $O(\ep^{-1/\nu})$ for reaching $\varphi(x_t)-\varphi^*\leq \ep$ \cite[]{Nes18}, which is known to be nearly optimal from complexity theory perspective \cite[]{GN15}.
When $g$ is strongly convex, \cite{Nes18} proposed the step size $\tau_t = \frac{6(t+1)}{(t+2)(2t+3)}$ for obtaining a better iteration complexity of $O(\ep^{-\frac{1}{2\nu}})$. Recently, \cite{Ghadimi19} improved this complexity to  $O(\ep^{-\frac{1-\nu}{2\nu}}\log\frac{1}{\ep})$ by using  a step size $\tau_t$ determined by a backtracking line search procedure. Remarkably, his method enjoys a linear 
rate of convergence when $\nu=1$. Ghadimi's method is also applicable to the case where $f$ is nonconvex and achieves an iteration complexity of $O(\ep^{-\frac{1-\nu}{2\nu}-1})$ for reaching $\delta_t \leq \ep$.

Conditional gradient methods were also studied for minimizing an $L$-smooth convex function over a strongly convex set \cite[e.g., see][]{LevPol,Dunn79,GH15}. Under the assumption that the set does not contain a stationary point of the function, it was established that the conditional gradient method with the step size  given in \eqref{eq:short-step size} has a linear rate of convergence \cite[]{LevPol,Dunn79}. Such a method was also generalized to minimize an $L$-smooth convex function over a uniformly convex set \cite[]{Ker21a}.

In this paper we first study a parameter-dependent conditional gradient method proposed in \cite[Algorithm 4]{ZF20} with a step size depending explicitly on the problem parameters  for solving a broad class of problems in the form of \eqref{main-prob}, including but not limited to the problems considered in the above references \cite[]{LevPol,Dunn79,Jaggi11,GH15,FG16,Lac16,Nes18,Ghadimi19,Ker21a,ZF20}. Though this method was analyzed in \cite[]{ZF20} for problem \eqref{main-prob} with \emph{convex} $f$ and \emph{bounded} $\dom{g}$, there is a lack of analysis for \eqref{main-prob} with nonconvex $f$. In this paper, we analyze the rate convergence of this method for problem \eqref{main-prob} with $f$ being possibly nonconvex under the assumption that $\nabla{f}$ is H\"older continuous and $\dom{g}$ is bounded  or the problem has a uniformly convex structure.   As a byproduct, we obtain a new iteration complexity of  $O(\ep^{-\frac{1-\nu}{2\nu}})$\,\footnote{For simplicity, the complexity here only emphasizes its dependence on the tolerance parameter $\ep$, while the other parameters such as $\nu$ are viewed as a fixed constant and thus omitted. Its detailed expression including the dependence on the other parameters can be found in \eqref{itrcompl-2-2} with $\rho=2$.} for problem \eqref{main-prob} with $\nabla{f}$ being H\"older continuous with exponent $\nu\in(0,1)$ and $g$ being strongly convex, which improves by the factor $\log(1/\ep)$ the previously best known one \cite[]{Ghadimi19}. 

Though the aforementioned parameter-dependent method (Algorithm~\ref{alg:FW-exact-ls}) is simple and also enjoys a nice iteration complexity, its step size may suffer from some issues. Indeed, its step size requires prior knowledge of the parameters $\nu$ and $M_\nu$ associated with the H\"older continuity of $\nabla f$. Since they depend on $f$, $g$, and also a particular norm on $\E$, they may be hard to be found if $f$ is sophisticated. On another hand, the parameters $\nu$ and $M_\nu$ are not unique. The tighter value of them typically leads to a faster convergent algorithm. Yet, it may be challenging to find the tightest possible value for them. Motivated by these, we further propose a parameter-free conditional gradient method (see  Algorithm~\ref{alg:main}) in which the step size is chosen by using a constructive local quadratic upper approximation and an adaptive line search scheme, without using prior knowledge of $\nu$ and $M_\nu$.
 We show that this method achieves the same rate of convergence as the parameter-dependent conditional gradient method, which, however, uses prior knowledge of the problem parameters $\nu$ and $M_\nu$.

The results of this paper were presented in the SIAM Conference on Optimization in July 2021 \cite[]{ILH21}. During the preparation of this paper, a concurrent work \cite[]{Pena22} proposed a parameter-free conditional gradient method for problem \eqref{main-prob} with \emph{convex} $f$ and established similar complexity bounds as the ones obtained in this paper yet in terms of primal-dual optimality gap that is typically weaker than the Frank-Wolfe gap which we use.  It shall be mentioned that our analyses are vastly different from those in \cite[]{Pena22} and shed new insights into conditional gradients for solving a broader class of problems.
 
The rest of this paper is organized as follows.
In Section~\ref{sec:pre} we introduce some notation and make some assumptions on the problem studied in this paper.
In Section~\ref{sec:exact-ls} we propose a parameter-dependent conditional gradient method, and show some results on its rate of convergence. In Section~\ref{sec:main} we propose a parameter-free conditional gradient method, and establish its rate of convergence and also iteration complexity.
Section~\ref{sec:num} presents numerical experiments to compare the performance of this method with the conditional gradient methods with some other step size rules. The proofs of main results are given in Section~\ref{sec:proofs}.
Finally, we make some concluding remarks in Section~\ref{sec:conclusion}.


\section{Notation and Assumptions}\label{sec:pre}

Throughout the paper, $\E$ is a finite dimensional real Hilbert space endowed with an inner product $\innprod{\cdot}{\cdot}$. Let $\norm{\cdot}$ be an arbitrary norm on $\E$ and $\|\cdot\|_*$ its dual, i.e., $\norm{z}_* = \sup_{\norm{x}\leq 1}\innprod{z}{x}$ for $z\in \E$.
We denote by $\RR_+$ and $\ZZ_+$ the set of nonnegative real numbers and the set of nonnegative integers, respectively. For any real number $a$, we denote by $a_+$ the nonnegative part of $a$, that is, $a_+=\max\{a,0\}$. 
For the convex function $g$, 
$\dom{g}$ denotes the domain of $g$, i.e., $\dom{g}=\{x\in\E: g(x) \neq +\infty\}$. We denote 
by $D_g$ the diameter of $\dom g$, that is, 
\begin{equation} \label{Dg}
D_g=\sup_{x,y \in \dom g}\norm{x-y}.
\end{equation}
Clearly, $D_g=+\infty$ if $\dom{g}$ is unbounded.

We make the following assumption on the functions $f$ and $g$ throughout this paper.


\begin{assumption}\label{assump:fg}
\begin{enumerate}[(i)]
\item 
The function $f:\E\to \RR\cup\{+\infty\}$ is a proper and lower-semicontinuous function.
Moreover, $f$ is differentiable on $\dom{g}$ and the gradient $\nabla{f}$ is H\"older continuous on $\dom g$, i.e., there exist $\nu \in (0,1]$ and $M_\nu>0$ such that
\begin{equation}\label{assump:Holder}
\norm{\nabla{f}(x)-\nabla{f}(y)}_* \leq M_\nu \norm{x-y}^\nu,\quad \forall x,y \in \dom g.
\end{equation}
\item 
The function $g:\E\to\RR\cup\{+\infty\}$ is a proper, lower-semicontinuous, and convex function. In addition, for each $x \in \dom{g}$,  the subproblem
\begin{equation}\label{eq:subproblem}
\min_{v\in \E}\{\innprod{\nabla{f}(x)}{v} + g(v)\}.
\end{equation}
has at least one optimal solution. 
\end{enumerate}
\end{assumption}

The function $f$ satisfying Assumption~\ref{assump:fg}(i) is commonly referred to as a {\it $(\nu,M_\nu)$-weakly smooth function} on $\dom g$. 
There are many instances of problem \eqref{main-prob} satisfying Assumption~\ref{assump:fg}(i). For example, problem \eqref{main-prob} with $\nabla f$ being 
 semi-algebraic\footnote{A \emph{semi-algebraic set} is a finite union of sets of the form $\{x\in\E~|~p_i(x)=0,\, i=1,\ldots,k,~ q_j(x)<0,\, j=1,\ldots,l\}$ for some real polynomials $p_i,q_j$. A map $h:\RR^m\to\RR^n$ is said to be \emph{semi-algebraic} if its graph $\{(x,y)~|~y=h(x)\}$ is a semi-algebraic set.} continuous and $\dom{g}$ being a compact semi-algebraic set is one of them \cite[][Proposition C.1]{Bolte20}. Also, there are some machine learning models satisfying Assumption~\ref{assump:fg}(i) \cite[see, e.g.,][]{Bolte20,Piotr16,TJVP20}.
In addition, Assumption~\ref{assump:fg}(ii) plays an important role in a conditional gradient method for solving problem \eqref{main-prob}. Indeed, all the conditional gradient methods in the literature solve a subproblem in the form of \eqref{eq:subproblem} per iteration.

In this paper, we will consider problem \eqref{main-prob} satisfying Assumption \ref{assump:fg} and one additional assumption that $g$ has a bounded domain or problem \eqref{main-prob} has a uniformly convex structure introduced below.


\begin{assumption}\label{assump:g}
Problem \eqref{main-prob} has a \emph{uniformly convex structure}, that is, there exist $\kappa> 0$ and $\rho \geq 2$ such that for any $x \in \dom{g}$ and any optimal solution $v^*$ to the subproblem \eqref{eq:subproblem}, we have
\begin{equation}\label{eq:uniconv-subprob-ineq}
\innprod{\nabla{f}(x)}{v}+g(v) - \innprod{\nabla{f}(x)}{v^*}-g(v^*) \geq \frac{\kappa}{\rho} \norm{v-v^*}^\rho, \quad \forall v \in \E.
\end{equation}
\end{assumption}


There are many instances of problem \eqref{main-prob} with a uniformly convex structure. We next provide two examples of \eqref{main-prob} for which Assumptions~\ref{assump:fg}  and \ref{assump:g} hold. 

\begin{example}[Optimization over a uniformly convex set]\label{ex:uniconv-set}
Let $C \subset \E$ be a nonempty compact \emph{$(c,\rho)$-uniformly convex set with respect to $\norm{\cdot}$} for some $c>0$ and $\rho\geq 2$,\footnote{For example, the $\ell_p$-balls are uniformly convex  with $\rho=\max(2,p)$ for $p \in (1,\infty)$.  In addition, $C$ is often referred to as a \emph{strongly convex set} if $\rho=2$. See \cite{Vial82,LevPol,GH15,Ker21a,Ker21} for the discussion on uniformly or strongly convex sets.} that is, 
$C$ is a nonempty compact convex set satisfying that $(1-\lambda) x + \lambda y + z \in C$ for any $x,y \in C$, $\lambda \in [0,1]$, and $z$ with $\norm{z} \leq \lambda(1-\lambda)\frac{c}{\rho}\norm{x-y}^\rho$.
Consider the problem 
\begin{equation} \label{example1}
\min_{x\in C}f(x),
\end{equation}
where $f$ is differentiable on $C$, $\nabla{f}$ is H\"older continuous on $C$, and $\alpha = \min_{x \in C}\norm{\nabla{f}(x)}_* > 0$ (i.e., no stationary point of $f$ belongs to $C$).
Let $g$ be the indicator function of $C$. One can easily observe that problem \eqref{example1} is a special case of \eqref{main-prob}, and moreover, Assumption~\ref{assump:fg} holds for it. We next verify that Assumption~\ref{assump:g} also holds for it. Indeed, let $x \in \dom g$ and $v^* \in \Argmin_{v \in \E}\{\innprod{\nabla{f}(x)}{v}+g(v)\}$ be arbitrarily chosen. Then we have   
 that $x\in C$ and $v^* \in \Argmin_{v \in C}\innprod{\nabla{f}(x)}{v}$. Since $C$ is a $(c,\rho)$-uniformly convex set, one has that $w=\lambda v^* + (1-\lambda) v +[ \lambda(1-\lambda)c\norm{v-v^*}^\rho/\rho] z \in C$ for any $\lambda \in (0,1)$, $v\in C$, and $z$ with $\norm{z}\leq 1$. This and the optimality of $v^*$ lead to 
\[ 
\innprod{\nabla f(x)}{\lambda v^* + (1-\lambda) v + \lambda(1-\lambda)\frac{c}{\rho}\norm{v-v^*}^\rho z-v^*}=\innprod{\nabla f(x)}{w-v^*}\geq 0,
\] 
which implies that $\innprod{\nabla{f}(x)}{v-v^*} \geq (\lambda c/\rho)\norm{v-v^*}^\rho\innprod{\nabla{f}(x)}{-z}$. Taking $\sup_{\norm{z}\leq 1}$ on both sides of this inequality, letting $\lambda \uparrow 1$, and using $\alpha = \min_{x \in C}\norm{\nabla{f}(x)}_*$,  we obtain that 
\[
\innprod{\nabla{f}(x)}{v-v^*}  \geq \norm{\nabla{f}(x)}_* \cdot \frac{c}{\rho} \norm{v-v^*}^\rho \ge \frac{\alpha c}{\rho} \norm{v-v^*}^\rho.
\]
This, together with $g$ being the indicator function of $C$, implies that \eqref{eq:uniconv-subprob-ineq} is satisfied with $\kappa=\alpha c$. Therefore, Assumption~\ref{assump:g} holds for problem \eqref{example1}.
\end{example}


\begin{example}[Optimization with a uniformly convex function]\label{ex:uniconv-func}
Consider a special case of problem \eqref{main-prob}, where $\nabla{f}$ is H\"older continuous on $\dom{g}$, and $g$ is a proper, lower-semicontinuous and \emph{$(\kappa,\rho)$-uniformly convex function with respect to $\norm{\cdot}$} for some $\kappa>0$ and $\rho\geq 2$, that is, $g$ satisfies that  
$$
g(\lambda x + (1-\lambda)y) \leq \lambda g(x) + (1-\lambda) g(y) - \lambda(1-\lambda)\frac{\kappa}{\rho}\norm{x-y}^\rho, \quad \forall x,y \in \dom g, \lambda \in [0,1].\footnote{For the case $\rho=2$, the function $g$ is a usual strongly convex function with modulus $\kappa$.}
$$
By this, one can observe that \cite[e.g., see][]{Zal}
\begin{equation} \label{g-lwrbnd}
g(y) \geq g(x) + g'(x;y-x) + \frac{\kappa}{\rho}\norm{x-y}^\rho,\quad \forall x,y \in \dom g,
\end{equation}
where $g'(x;d) = \lim_{s\downarrow 0}(g(x+s d)-g(x))/s$ is the directional derivative of $g$ along the direction $d$. It then follows that $g$ is coercive, which together with the lower-semicontinuity of $g$ implies that subproblem \eqref{eq:subproblem} has at least one optimal solution. Thus, Assumption~\ref{assump:fg} holds for this problem. We next verify that Assumption~\ref{assump:g} also holds for it. Indeed, let $x \in \dom g$ and $v^*= \argmin_{v \in \E}\{\innprod{\nabla{f}(x)}{v}+g(v)\}$. By these and \eqref{g-lwrbnd}, one has
\[
\innprod{\nabla{f}(x)}{v}+g(v) \ge \innprod{\nabla{f}(x)}{v^*} + \innprod{\nabla{f}(x)}{v-v^*} +g(v^*)+ g'(v^*;v-v^*) + \frac{\kappa}{\rho}\norm{v-v^*}^\rho,
\]
for all $v \in \dom{g}$.
In addition, by the optimality of $v^*$, we have $\innprod{\nabla{f}(x)}{v-v^*} + g'(v^*;v-v^*) \ge 0$ for all $v \in \dom g$. These two inequalities immediately yield \eqref{eq:uniconv-subprob-ineq}. Therefore, Assumption~\ref{assump:g} also holds for this problem.
\end{example}


\section{A Parameter-Dependent Conditional Gradient Method}
\label{sec:exact-ls}

In this section we present in Algorithm~\ref{alg:FW-exact-ls} a parameter-dependent conditional gradient method for solving problem \eqref{main-prob}, whose step size $\tau_t$ depends on the problem parameters $\nu$ and $M_\nu$ explicitly. This algorithm was proposed in \cite[Algorithm 4]{ZF20} and analyzed by them for the case where $f$ is convex and $\dom{g}$ is bounded. However, it was not analyzed for the case where $f$ is nonconvex. In what follows, we will analyze its rate of convergence for solving \eqref{main-prob} with $f$ being possibly nonconvex under the assumption that $\dom{g}$ is bounded or problem \eqref{main-prob} has a uniformly convex structure. It shall be mentioned that the convergence results established in this section also hold for a variant of Algorithm~\ref{alg:FW-exact-ls}  with the exact step size $\tau_t \in \Argmin_{\tau \in [0,1]}\varphi((1-\tau)x_t + \tau v_t)$.


\begin{algorithm}[t]
\caption{A parameter-dependent conditional gradient method}
\label{alg:FW-exact-ls}
\textbf{Input:} $x_0 \in \dom{g}$.
\begin{algorithmic}[1]
\For{$t=0,1,2,\ldots,$}
	\State $v_t \in \Argmin_{x \in \E}\{\innprod{\nabla{f}(x_t)}{x} + g(x)\}$.
	\State $\delta_t = \innprod{\nabla{f}(x_t)}{x_t} + g(x_t) - \innprod{\nabla{f}(x_t)}{v_t} - g(v_t)$.
	\State $\tau_t = \min\left\{1,\left(\frac{\delta_t}{M_\nu\norm{x_t-v_t}^{1+\nu}}\right)^{\frac{1}{\nu}}\right\}$.
	\State $x_{t+1} =(1-\tau_t)x_t + \tau_t v_t$.
\EndFor
\end{algorithmic}
\end{algorithm}


Before proceeding, we state a well-known lemma \cite[e.g., see][]{Ghadimi19}, which shows that the quantity $\delta_t$, commonly referred to as the \emph{Frank-Wolfe gap}, provides an upper bound on the optimality gap of problem \eqref{main-prob}  at $x_t$ when $f$ is convex.  For the sake of completeness, we include a proof for it. 

\begin{lemma}\label{lem:FW-gap-bound-convex}
Let the sequences $\{x_t\}$ and $\{\delta_t\}$ be generated in Algorithm~\ref{alg:FW-exact-ls}.  
Suppose that $f$ is convex. Then it holds that $\delta_t \geq \varphi(x_t)-\varphi^*$ for all $t\geq 0$.
\end{lemma}

\begin{proof}
Let $x^*$ be an arbitrary optimal solution of problem \eqref{main-prob}. Then $f(x^*)+g(x^*)= \varphi^*$. By this, the expression of $\delta_t$, and the convexity of $f$, we have that for all $t\geq 0$, 
\begin{align}
\delta_t &= \innprod{\nabla{f}(x_t)}{x_t} + g(x_t) - \innprod{\nabla{f}(x_t)}{v_t} - g(v_t) \nonumber\\
& \geq \innprod{\nabla{f}(x_t)}{x_t} + g(x_t) - \innprod{\nabla{f}(x_t)}{x^*} - g(x^*) \nonumber\\
&\geq f(x_t) + g(x_t) -f(x^*)-g(x^*) = \varphi(x_t)-\varphi^*. \label{eq:FW-gap-bound-convex}
\end{align}
\end{proof}


\begin{remark}
From the proof of Lemma~\ref{lem:FW-gap-bound-convex}, one can observe that the convexity of $f$ is only used  in \eqref{eq:FW-gap-bound-convex}. More generally, \eqref{eq:FW-gap-bound-convex} is also valid if $f$ satisfies the \emph{star-convexity} property \cite[]{NesPol06}: there exists some $x^* \in \Argmin_x \varphi(x)$ such that $f(\lambda x^*+(1-\lambda)x) \leq \lambda f(x^*) + (1-\lambda) f(x)$ for all $\lambda \in [0,1]$ and $x \in \dom g$. Thus, the conclusion of Lemma~\ref{lem:FW-gap-bound-convex} also holds if $f$ satisfies the star-convexity property. Moreover, all the results established in this paper for a convex $f$ also hold for a star-convex $f$.
\end{remark}

In what follows, we state some results regarding the rate of convergence of Algorithm~\ref{alg:FW-exact-ls} in Theorems~\ref{rate-exact-ls-1} and \ref{rate-exact-ls-2}, whose proofs are deferred to Section~\ref{sec:proof-1}. In particular, we first present the results under the assumption that $g$ has a bounded domain, namely, $D_g<+\infty$.


\begin{theorem}\label{rate-exact-ls-1}
Let the sequences $\{x_t\}$ and $\{\delta_t\}$ be generated in Algorithm~\ref{alg:FW-exact-ls}.  
Suppose that Assumption~\ref{assump:fg} holds, $D_g < +\infty$, and that $\delta_t> 0$ for all $t \ge 0$, where $D_g$ is defined in \eqref{Dg}.  Let $\delta_t^*= \min_{0\leq i\leq t}\delta_i$ and
\begin{align}
&
A = M_\nu^{\frac{1}{\nu}}D_g^{\frac{1+\nu}{\nu}}, \quad  t_0=\left\lceil \frac{1+\nu}{\nu}\left(\log \frac{(1+\nu)(\varphi(x_0)-\varphi^*)}{\nu A^{\nu}}\right)_+ \right\rceil, \nonumber \\
& \overline\gamma_t=[(\varphi(x_{t_0})-\varphi^*)^{-\frac{1}{\nu}}+(1+\nu)^{-1}A^{-1}(t-t_0)]^{-\nu}, \quad \forall t \ge t_0.\nonumber
\end{align}
Then the following statements hold.
\begin{itemize}
\item[(i)] $\{\varphi(x_t)\}$ is non-increasing and $\varphi_* =\lim_{t\to \infty}\varphi(x_t)$ exists. In addition, $\{\delta_t^*\}$ satisfies 
\begin{equation}\label{delta*-uppbnd-1}
\delta_t^*  \leq \max\left\{\frac{(1+\nu)(\varphi(x_0)-\varphi_*)}{\nu(t+1)},~
	\left(\frac{(1+\nu)A(\varphi(x_0)-\varphi_*)}{\nu(t+1)}\right)^{\frac{\nu}{1+\nu}}\right\}, \quad \forall t \ge 0.
\end{equation}
\item[(ii)] Assume additionally that $f$ is convex.
Then we have
\begin{align*}
&
\varphi(x_t)-\varphi^* \leq\overline\gamma_t, \quad \forall t\geq t_0,\\
&
\delta_t^* \leq e^{\frac{1}{e}} \overline\gamma_{\lfloor (t+t_0+1)/2 \rfloor},
\quad \forall t \geq t_0 + \frac{2(1+\nu)A}{\nu(\varphi(x_{t_0})-\varphi^*)^{\frac{1}{\nu}}}.
\end{align*}
\end{itemize}
\end{theorem}


\begin{remark}
When $f$ is nonconvex, the limit $\varphi_*=\lim_{t\to \infty}\varphi(x_t)$  in Theorem \ref{rate-exact-ls-1} can be interpreted as the function value at some stationary point of $\varphi$. Indeed, for any convergent subsequence $\{x_t\}_{t \in T}$ with limit $x_*$ such that $\{\delta_t\}_{t \in T} \to 0$, it follows from \eqref{def-FW-gap} that
$$
\delta_t \geq \innprod{\nabla f(x_t)}{x_t-x} + g(x_t)-g(x),\quad \forall x \in \E.
$$
Taking the limit over $t \in T$, we can see that
$x_*\in\Argmin_x\{\innprod{\nabla f(x_*)}{x} + g(x)\}$. Thus, $x_*$ is a stationary point of $\varphi$ and moreover $\varphi_* = \varphi(x_*)$.
\end{remark}


We next present some results regarding the rate of convergence of Algorithm~\ref{alg:FW-exact-ls} under the assumption that problem \eqref{main-prob} has a uniformly convex structure, namely, Assumption~\ref{assump:g} holds.

\begin{theorem}\label{rate-exact-ls-2}
Let the sequences $\{x_t\}$ and $\{\delta_t\}$ be generated in Algorithm~\ref{alg:FW-exact-ls}.  
Suppose that Assumptions~\ref{assump:fg} and \ref{assump:g} hold and that $\delta_t> 0$ for all $t \ge 0$. Let $\delta_t^*= \min_{0\leq i\leq t}\delta_i$ for all $t \ge 0$ and
\begin{align}
&
A = \left(\frac{\rho}{\kappa}\right)^{\frac{1+\nu}{\rho\nu}}M_\nu^{\frac{1}{\nu}}, \quad  t_0= \left\lceil\frac{1+\nu}{\nu}\left(\log\frac{(1+\nu)(\varphi(x_0)-\varphi^*)}{\nu A^{\frac{\rho \nu}{\rho-1-\nu}}} \right)_+ \right\rceil,  \nonumber\\
& \overline\gamma_t=\left\{
\begin{array}{ll}
(\varphi(x_0)-\varphi^*) \exp\left(-\frac12\min\{1,\frac{\kappa}{2M_1}\}t\right) & \text{if} \ \nu=1 \text{ and } \rho=2, \\[2truemm]
 \left[(\varphi(x_{t_0})-\varphi^*)^{-\frac{\rho-1-\nu}{\rho\nu}}+\frac{\rho-1-\nu}{\rho(1+\nu)}A^{-1}(t-t_0) \right]^{-\frac{\rho\nu}{\rho-1-\nu}} & \text{otherwise}.
\end{array}
\right.   \nonumber
\end{align}
Then the following statements hold.
\begin{itemize}
\item[(i)] $\{\varphi(x_t)\}$ is non-increasing and $\varphi_*=\lim_{t\to \infty}\varphi(x_t)$ exists.  In addition, $\{\delta_t^*\}$ satisfies 
\begin{equation}\label{delta*-uppbnd-2}
\delta_t^*  \leq \max\left\{\frac{(1+\nu)(\varphi(x_0)-\varphi_*)}{\nu(t+1)},\left(\frac{(1+\nu)A(\varphi(x_0)-\varphi_*)}{\nu(t+1)}\right)^{\frac{\rho\nu}{(\rho-1)(1+\nu)}}
\right\},
\end{equation}
for all $t\geq 0$.
\item[(ii)] Assume additionally that $f$ is convex.
\begin{enumerate}[(a)]
\item
When $\nu=1$ and $\rho=2$, we have 
\begin{align*}
&\varphi(x_t)-\varphi^* \leq \overline{\gamma}_t,\quad \forall t \geq 0,\\
&
\delta_t^* \leq e^{\frac{1}{e}}\overline{\gamma}_{\left\lfloor (t+2)/2 \right\rfloor}, \quad \forall t \geq 4\max\left\{1,\frac{2M_1}{\kappa}\right\}.
\end{align*}
\item
When $\nu \neq 1$ or $\rho\neq 2$, we have
\begin{align*}
& \varphi(x_t)-\varphi^* \leq \overline\gamma_t,
\quad \forall t \geq t_0,\\
&\delta_t^* \leq e^{\frac{1}{e}}\overline\gamma_{\lfloor (t+t_0+1)/2\rfloor},
\quad \forall t \geq t_0+\frac{2(1+\nu)A}{\nu (\varphi(x_{t_0})-\varphi^*)^{\frac{\rho\nu}{\rho-1-\nu}} }.
\end{align*}
\end{enumerate}
\end{itemize}
\end{theorem}

\begin{remark}
It can be observed from Theorem \ref{rate-exact-ls-2} that under Assumptions~\ref{assump:fg} and \ref{assump:g}, Algorithm~\ref{alg:FW-exact-ls} enjoys a linear rate of convergence when applied to problem \eqref{main-prob} with $f$ being convex, $\nu=1$ and $\rho=2$.
\end{remark}

\section{A Parameter-Free Conditional Gradient Method}  \label{sec:main}

As seen from above, Algorithm~\ref{alg:FW-exact-ls} is not only simple but also enjoys a nice rate of convergence. However, its step size $\tau_t$ may suffer from some practical issues. Indeed, to evaluate $\tau_t$, one needs to know the problem parameters $\nu$ and $M_\nu$ in advance. As observed from \eqref{assump:Holder}, these parameters depend on $f$, $g$, and also a particular norm on $\E$. Thus, it may not be easy to find them if $f$ is a sophisticated function. On another hand, the parameters $\nu$ and $M_\nu$ are not unique. The tighter value of them typically leads to a faster convergent algorithm. Yet, it may be challenging to find the tightest possible value for them. 
Motivated by these, we next propose a parameter-free conditional gradient method (Algorithm~\ref{alg:main}) in which the step size is chosen by using a constructive local quadratic upper approximation and an adaptive line search scheme, without using prior knowledge of $\nu$ and $M_\nu$.

%

\begin{algorithm}[t]
\caption{
A parameter-free conditional gradient method
}
\label{alg:main}
\textbf{Input:} $x_0 \in \dom{g}$ and $L_{-1} > 0$.
\begin{algorithmic}[1]
\For{$t=0,1,2,\ldots,$}
	\State $v_t \in \Argmin_{x \in \E}\{\innprod{\nabla{f}(x_t)}{x} + g(x)\}$.
	\State $\delta_t = \innprod{\nabla{f}(x_t)}{x_t}+g(x_t) - \innprod{\nabla{f}(x_t)}{v_t}-g(v_t)$.
	\Repeat{ \textbf{for} $i=0,1,2,\ldots,$} 
		\State $L_t^{(i)}=2^{i-1} L_{t-1}$.
		\State $\tau_t^{(i)}=\min\left\{1, \frac{\delta_t}{2 L_t^{(i)}\norm{x_t-v_t}^2}\right\}$.
		\State $x_{t+1}^{(i)}=(1-\tau_t^{(i)})x_t + \tau_t^{(i)} v_t$.
	\Until{
\begin{equation}\label{phi-reduct}
\varphi(x_{t+1}^{(i)}) \leq \varphi(x_t)-\frac{1}{2}\tau_t^{(i)}\delta_t + \frac{1}{2}L_t^{(i)}(\tau_t^{(i)})^2\norm{x_t-v_t}^2.
\end{equation}
	} 
	\State Set $(x_{t+1},L_t,\tau_t)\leftarrow (x_{t+1}^{(i)}, L_t^{(i)}, \tau_t^{(i)})$.
\EndFor
\end{algorithmic}
\end{algorithm}

We now provide some explanation for Algorithm~\ref{alg:main}. Observe that the step size $\tau_t$ in Algorithm~\ref{alg:FW-exact-ls} is the minimizer of $h_t(\tau)= \varphi(x_t) -\tau \delta_t + \tau^{1+\nu} \frac{M_\nu}{1+\nu}\norm{x_t-v_t}^{1+\nu}$ over $[0,1]$. 
Assuming that 
$\nu$ and $M_\nu$ are unknown, we can find a quadratic approximation to  
$h_t$, without explicitly involving $\nu$ and $M_\nu$, and then obtain a step size by minimizing it over $[0,1]$. Indeed, by the H\"older continuity of $\nabla{f}$, the following inequality holds \citep[e.g, see][Lemma~2]{Nes15}:
\begin{equation}\label{eq:Holder-approx}
f(y) \le  f(x) + \innprod{\nabla{f}(x)}{y-x} + \frac{L(\ep)}{2}\norm{x-y}^2 + \ep,\quad \forall x,y \in \dom g,~\forall \ep>0,\footnote{By convention, we set $0^0=1$. One can observe that when $\nu=1$, $L(\ep)$ becomes $M_\nu$ and thus \eqref{eq:Holder-approx} still holds.}
\end{equation}
where
\begin{equation}\label{eq:L-ep}
L(\ep)=\left(\frac{1-\nu}{1+\nu}\cdot \frac{1}{2\ep}\right)^{\frac{1-\nu}{1+\nu}} M_\nu^{\frac{2}{1+\nu}}, \quad  \forall \ep>0.
\end{equation}
By \eqref{eq:Holder-approx} and a suitable choice of $\ep$ (see the proof of next theorem for details), one can obtain a quadratic approximation to $h_t$ given by 
\[
 \tilde h_t (\tau) = \varphi(x_t)-\frac{1}{2}\tau \delta_t + \frac{1}{2}L_t\tau^2\norm{x_t-v_t}^2
\]
for some $L_t>0$, which is determined by an adaptive line search scheme without explicitly using $\nu$ and $M_\nu$. The step size $\tau_t$ in Algorithm~\ref{alg:main} is then obtained by minimizing $\tilde h_t$ in place of $h_t$ over $[0,1]$. Therefore, Algorithm~\ref{alg:main} does not explicitly use the parameters $\nu$ and $M_\nu$.


The following theorem shows that Algorithm~\ref{alg:main} is well-defined, whose proof is deferred to Section~\ref{sec:proof-2}. In particular, we will establish that in each outer iteration of  Algorithm~\ref{alg:main} the adaptive line search procedure must terminate after a finite number of trials. 
We will also establish some other properties for the adaptive line search procedure.

\begin{theorem}\label{lem:L-est}
Let the sequences $\{L_t\}$ and $\{\delta_t\}$ be generated in Algorithm~\ref{alg:main}. Suppose that Assumption~\ref{assump:fg} holds and that $\delta_t> 0$ for all $t \ge 0$.\footnote{If $\delta_t=0$ for some $t \ge 0$, $x_t$ is already a stationary point of problem \eqref{main-prob} and Algorithm~\ref{alg:main} shall be terminated.}  Let 
\begin{equation}\label{eq:def-tilde-Lt}
\widetilde{L}_t=\max\left\{L(\delta_t/2), L\left(\frac{\delta_t^2}{4\norm{x_t-v_t}^2}\right)^{\frac{1+\nu}{2\nu}}\right\}
\end{equation}
for all $t \ge 0$, where $L(\cdot)$ is defined in \eqref{eq:L-ep}. For any $\delta>0$, let 
\begin{equation}\label{eq:def-tilde-L}
\overline{L}(\delta) = \left\{
\begin{array}{ll}
\max\left\{
\left(\frac{1-\nu}{1+\nu}\frac{1}{\delta}\right)^{\frac{1-\nu}{1+\nu}}M_\nu^{\frac{2}{1+\nu}},~~
\left(\frac{2(1-\nu)}{1+\nu}\right)^{\frac{1-\nu}{2\nu}}
M_\nu^{\frac{1}{\nu}}
\left(\frac{D_g}{\delta}\right)^{\frac{1-\nu}{\nu}}\right\} & \text{if } \dom{g} \text{ is bounded}, \\
\max\left\{
\left(\frac{1-\nu}{1+\nu}\frac{1}{\delta}\right)^{\frac{1-\nu}{1+\nu}}M_\nu^{\frac{2}{1+\nu}},~~
\left(\frac{2(1-\nu)}{1+\nu}\right)^{\frac{1-\nu}{2\nu}}
M_\nu^{\frac{1}{\nu}}
\left(\frac{\rho}{\kappa \delta^{\rho-1}}\right)^{\frac{1-\nu}{\rho\nu}} \right\} & \text{if Assumption~\ref{assump:g} holds}. 
\end{array}
\right. 
\end{equation}
Then the following statements hold. 
\begin{enumerate}[(i)]
\item
The inequality \eqref{phi-reduct} holds whenever $L_t^{(i)} \geq \widetilde{L}_t$.
\item $L_t \leq 2\max_{0\leq i\leq t} \widetilde{L}_i$ holds for any 
$t \geq (\log_2 (L_{-1}/\widetilde{L}_0))_+$.
\item
Suppose further that $\min_{0\leq i \leq t}\delta_i \geq \ep$ for some $t \ge 0$ and $\ep>0$.
Then the total number of inner iterations performed by the adaptive line search procedure until the $t$-th  iteration of Algorithm~\ref{alg:main} is bounded by
$2t +2+ [\log_2 (2\overline{L}(\ep) /L_{-1})]_+$.
\end{enumerate}
\end{theorem}

%

The next theorem establishes some results for the case where $g$ has a bounded domain, namely, $D_g<+\infty$, whose proof is deferred to Section~\ref{sec:proof-2}.

\begin{theorem}\label{th:main1}
Let the sequences $\{x_t\}$ and $\{\delta_t\}$ be generated in Algorithm~\ref{alg:main}.  
Suppose that Assumption~\ref{assump:fg} holds, $D_g < +\infty$, and that $\delta_t>0$ for all $t \ge 0$, where $D_g$ is defined in \eqref{Dg}.  Let $\delta_t^*= \min_{0\leq i\leq t}\delta_i$ for all $t \ge 0$ and
\begin{align}
&
A = (2M_\nu)^{\frac{1}{\nu}}D_g^{\frac{1+\nu}{\nu}}, \quad \tilde t_0=\left\lceil\left(\log_2 \frac{L_{-1}}{\widetilde{L}_0}\right)_+\right\rceil,
\quad
t_0=\left\lceil 4\left(\log\frac{4(\varphi(x_{\tilde{t}_0})-\varphi^*)}{A^\nu}\right)_+ \right\rceil,
\nonumber\\
&
\overline\gamma_t=\left[(\varphi(x_{t_0+\tilde{t}_0})-\varphi^*)^{-\frac{1}{\nu}}+(4\nu A)^{-1}(t - t_0) \right]^{-\nu}, \quad \forall t \ge t_0, \nonumber 
\end{align}
where $\widetilde{L}_0$ is defined in \eqref{eq:def-tilde-Lt}.
Then the following statements hold.
\begin{itemize}
\item[(i)] $\{\varphi(x_t)\}$ is non-increasing and $\varphi_* =\lim_{t\to \infty}\varphi(x_t)$ exists. In addition, $\{\delta_t^*\}$ satisfies 
\begin{equation}\label{delta*-uppbnd-3}
\delta^*_t
\leq
\max\left\{\frac{4(\varphi(x_{\tilde t_0})-\varphi_*)}{t+1-\tilde t_0},
	\left(\frac{4 A (\varphi(x_{\tilde t_0})-\varphi_*)}{t+1-\tilde t_0}\right)^{\frac{\nu}{1+\nu}}\right\}, \quad \forall t \ge \tilde t_0.
\end{equation}
\item[(ii)] Assume additionally that $f$ is convex.
Then we have
\begin{align*}
&
\varphi(x_t)-\varphi^* \leq\overline\gamma_{t-\tilde t_0}, \quad \forall t \ge\tilde t_0+ t_0,\\
&
\delta^*_t \leq e^{\frac{1}{e}} \overline\gamma_{\lfloor (t-\tilde t_0+t_0+1)/2 \rfloor},\quad \forall t\geq \tilde t_0+t_0+\frac{8A}{(\varphi(x_{\tilde{t}_0+t_0})-\varphi^*)^{\frac{1}{\nu}}}.
\end{align*}
\end{itemize}
\end{theorem}

%

As an immediate consequence of Theorem \ref{th:main1}, we obtain the following complexity results for Algorithm~\ref{alg:main} for finding an approximate solution of problem \eqref{main-prob} with an $\varepsilon$-Frank-Wolfe gap, whose proofs are omitted.

\begin{corollary} \label{cor:itrcompl-1}
Under the same settings as in Theorem~\ref{th:main1}, Algorithm~\ref{alg:main} reaches the criterion 
$\delta_t \le \ep$ within  
\begin{equation*}
\tilde t_0 + \frac{4(\varphi(x_{\tilde t_0})-\varphi_*)}{\ep}\max\left\{1,\left(\frac{2M_\nu D_g^{1+\nu}}{\ep}\right) ^{\frac{1}{\nu}}\right\}
\end{equation*}
iterations. Furthermore, if $f$ is convex, it reaches the criterion 
$\delta_t \le \ep$ within 
\begin{equation*}
\tilde t_0 + t_0
+
8\left(\frac{2M_\nu D_g^{1+\nu}}{\varphi(x_{\tilde t_0 + t_0})-\varphi^*}\right)^{\frac{1}{\nu}} \max\left\{
	1,
	\nu  \left[
	\left( \frac{ e^{\frac{1}{e}} (\varphi(x_{\tilde t_0 + t_0})-\varphi^*) }{\ep} \right)^{\frac{1}{\nu}}
	-1
	\right]
\right\}
\end{equation*}
iterations.
\end{corollary}

%

In what follows, we present some results regarding the rate of convergence of Algorithm~\ref{alg:main} for the case where  problem \eqref{main-prob} has a uniformly convex structure, namely,  Assumption~\ref{assump:g} holds, whose proof is deferred to Section~\ref{sec:proof-2}.

\begin{theorem}\label{th:main2}
Let the sequences $\{x_t\}$ and $\{\delta_t\}$ be generated by Algorithm~\ref{alg:main}.  
Suppose that Assumptions~\ref{assump:fg} and \ref{assump:g} hold and that $\delta_t>0$ for all $t \ge 0$. Let $\delta_t^*= \min_{0\leq i\leq t}\delta_i$ for all $t \ge 0$, and 
\begin{align}
&
A = \left(\frac{\rho}{\kappa}\right)^{\frac{1+\nu}{\rho\nu}}(2M_\nu)^{\frac{1}{\nu}},
\quad
\tilde t_0=\left\lceil\left(\log_2 \frac{L_{-1}}{\widetilde{L}_0}\right)_+\right\rceil,
\quad
t_0=\left\lceil 4\left(\log\frac{4(\varphi(x_{\tilde{t}_0})-\varphi^*)}{A^\frac{\rho\nu}{\rho-1-\nu}}\right)_+ \right\rceil, \nonumber\\
&\overline\gamma_t=\left\{
\begin{array}{ll}
(\varphi(x_{\tilde t_0})-\varphi^*) \exp\left(-\frac14\min\{1,\frac{\kappa}{4M_1}\}t\right) & \text{if} \ \nu=1 \text{ and } \rho=2, \\[2truemm]
\left[(\varphi(x_{\tilde t_0 + t_0})-\varphi^*)^{-\frac{\rho-1-\nu}{\rho\nu}}+\frac{\rho-1-\nu}{4\rho\nu}A^{-1}(t-t_0) \right]^{-\frac{\rho\nu}{\rho-1-\nu}} & \text{otherwise}, 
\end{array}
\right.   \nonumber
\end{align}
where $\widetilde{L}_0$ is defined in \eqref{eq:def-tilde-Lt}.
Then the following statements hold.
\begin{itemize}
\item[(i)] $\{\varphi(x_t)\}$ is non-increasing and $\varphi_* =\lim_{t\to \infty}\varphi(x_t)$ exists. In addition, $\{\delta_t^*\}$ satisfies 
\begin{equation}\label{delta*-uppbnd-4}
\delta^*_t  \leq \max\left\{\frac{4(\varphi(x_{\tilde t_0})-\varphi_*)}{t+1-\tilde t_0},
	\left(\frac{4A(\varphi(x_{\tilde t_0})-\varphi_*)}{t+1-\tilde t_0}\right)^{\frac{\rho\nu}{(\rho-1)(1+\nu)}}\right\},
\end{equation}
for all $t \ge \tilde t_0$.
\item[(ii)] Assume additionally that $f$ is convex.
\begin{enumerate}[(a)]
\item
When $\nu=1$ and $\rho=2$, we have 
\begin{align*}
&\varphi(x_t)-\varphi^* \leq \overline{\gamma}_{t-\tilde t_0},\quad \forall t \geq \tilde t_0,\\
&
\delta_t^* \leq e^{\frac{1}{e}}\overline{\gamma}_{\left\lfloor (t-\tilde t_0+2)/2 \right\rfloor}, \quad \forall t \geq \tilde t_0 + 4\max\left\{1,\frac{2M_1}{\kappa}\right\}.
\end{align*}
\item
When $\nu \neq 1$ or $\rho\neq 2$, we have
\begin{align*}
&
\varphi(x_t)-\varphi^*
\leq
\overline\gamma_{t-\tilde t_0},\quad \forall t \ge\tilde t_0+ t_0,\\
&
\delta_t^*
\leq
e^{\frac{1}{e}} \overline\gamma_{\lfloor (t-\tilde t_0+t_0+1)/2 \rfloor},
\quad
\forall t \geq \tilde t_0 + t_0 + \frac{8A}{(\varphi(x_{\tilde t_0 + t_0})-\varphi^*)^{\frac{\rho-1-\nu}{\rho\nu}}}.
\end{align*}
\end{enumerate}
\end{itemize}
\end{theorem}

As an immediate consequence of Theorem \ref{th:main2}, we obtain the following complexity results for Algorithm~\ref{alg:main} for finding an approximate solution of problem \eqref{main-prob} with an $\varepsilon$-Frank-Wolfe gap, whose proofs are omitted.

%

\begin{corollary} \label{cor:itrcompl-2}
Under the same settings as in Theorem~\ref{th:main2}, Algorithm~\ref{alg:main} reaches the criterion 
$\delta_t \le \ep$ within  
\begin{equation*}
\tilde t_0 + \frac{4(\varphi(x_{\tilde t_0})-\varphi_*)}{\ep}
\max\left\{1,
	\frac{\rho^{\frac{1+\nu}{\rho\nu}}(2M_\nu)^{\frac{1}{\nu}}}{\kappa^{\frac{1+\nu}{\rho\nu}}\ep^{\frac{\rho-1-\nu}{\rho\nu}}}
\right\}
\end{equation*}
iterations. Furthermore, if $f$ is convex, $\nu=1$ and $\rho=2$, Algorithm~\ref{alg:main} reaches the criterion 
$\delta_t \le \ep$ within 
\begin{equation}\label{itrcompl-2-lin}
\tilde t_0 + 8 \max\left\{1,\frac{4M_1}{\kappa}\right\} \max \left\{1, \log \frac{\varphi(x_{\tilde t_0})-\varphi^*}{\ep}\right\}
\end{equation}
iterations. In addition, if $f$ is convex, and $\nu \neq 1$ or $\rho \neq 2$, Algorithm~\ref{alg:main} reaches the criterion 
$\delta_t \le \ep$ within 
\begin{equation}\label{itrcompl-2-2}
\tilde t_0 + t_0 +
\frac{8\rho^{\frac{1+\nu}{\rho\nu}}(2M_\nu)^{\frac{1}{\nu}}}{\kappa^{\frac{1+\nu}{\rho\nu}}(\varphi(x_{\tilde t_0 + t_0})-\varphi^*)^{\frac{\rho-1-\nu}{\rho\nu}}}
\max \left\{
	1,
	\frac{\rho\nu}{\rho-1-\nu} \left[
	\left(
		\frac{e^{\frac{1}{e}} (\varphi(x_{\tilde t_0 + t_0})-\varphi^*) }{\ep}
	\right)^{\frac{\rho-1-\nu}{\rho\nu}}
	-1
	\right]
\right\}
\end{equation}
iterations.
\end{corollary}


\begin{remark}
In view of the identity $\lim_{\alpha \to 0}\frac{1}{\alpha}(x^\alpha-1)=\log x$, one can observe that the limit of \eqref{itrcompl-2-2} as $\nu \to 1$ and $\rho\to 2$ is 
\[
O\left(\frac{M_1}{\kappa} \max \left\{1, \log \frac{\varphi(x_{\tilde t_0})-\varphi^*}{\ep}\right\}\right),
\]
which is consistent with the bound \eqref{itrcompl-2-lin} for the case with $\nu=1$ and $\rho=2$.
\end{remark}

\subsection{Iteration Complexity}\label{sec:iter-compl}

As mentioned earlier, the Frank-Wolfe gap $\delta_t$ defined in \eqref{def-FW-gap} is a computable quantity and can be used to measure whether the associated iterate $x_t$ is an approximate stationary point of problem \eqref{main-prob}. Therefore, $\delta_t \leq \ep$ 
can be used as a practical termination criterion for Algorithms~\ref{alg:FW-exact-ls} and \ref{alg:main} for a prescribed tolerance $\ep>0$. 
Besides, one can observe from Theorems \ref{rate-exact-ls-1}, \ref{rate-exact-ls-2}, \ref{th:main1} and \ref{th:main2} that Algorithms~\ref{alg:FW-exact-ls} and \ref{alg:main} enjoy the same rate of convergence and thus the same iteration complexity with respect to the termination criterion $\delta_t \leq \ep$. 
Consequently, it suffices to discuss the iteration complexity of Algorithm \ref{alg:main} with such a termination criterion.

One can observe from Corollaries~\ref{cor:itrcompl-1} and \ref{cor:itrcompl-2} that the iteration complexity of Algorithm~\ref{alg:main} for reaching the termination criterion $\delta_t \le \ep$ is:
\begin{enumerate}[(i)]
\item $O(\ep^{-1-1/\nu})$  if $f$ is nonconvex and $\dom{g}$ is bounded;
\item $O(\ep^{-1-(\rho-1-\nu)/(\rho\nu)})$  if $f$ is nonconvex and problem \eqref{main-prob} has a uniformly convex structure;
\item $O(\ep^{-1/\nu})$ if $f$ is convex and $\dom{g}$ is bounded; 
\item $O(\log(1/\ep))$  if $f$ is convex and  problem \eqref{main-prob} has a uniformly convex structure with $\nu=1$ and $\rho=2$;
\item $O(\ep^{-(\rho-1-\nu)/(\rho\nu)})$  if $f$ is convex and problem \eqref{main-prob} has a uniformly convex structure  with $\nu \neq 1$ or $\rho \neq 2$.
\end{enumerate}

Since $\rho\geq 2$ and $\nu\in(0,1]$, one has $\ep^{-(\rho-1-\nu)/(\rho\nu)}< \ep^{-1/\nu}$ and $\ep^{-1-(\rho-1-\nu)/(\rho\nu)}< \ep^{-1-1/\nu}$ when $\nu\neq 1$ or $\rho\neq 2$. In view of this and the above complexity results, we can observe that Algorithm~\ref{alg:main} enjoys a lower iteration complexity bound under Assumption \ref{assump:g} than the one under the assumption that $\dom{g}$ is bounded. Besides, the iteration complexity bound in (iii) matches the ones obtained in \cite[]{Nes15,ZF20}. It should, however, be noted that the conditional gradient methods in \cite[]{Nes15,ZF20} use the step size $\tau_t=2/(t+2)$ and the same one as given in Algorithm \ref{alg:FW-exact-ls}, respectively. 
These step sizes are not locally adaptive because they use none of local or global problem information, and could be conservative in practice.
Moreover, the latter one requires prior knowledge of the parameters $\nu$ and $M_\nu$. In contrast with them, the step size in Algorithm~\ref{alg:main} is locally adaptive and free of problem parameters. In addition, the iteration complexity bounds in (ii) and (iv) match the ones obtained in \cite[]{Ghadimi19} for the case with $\rho=2$.
The iteration complexity bound in (v) improves by the factor $\log(1/\ep)$ the one established in \cite[]{Ghadimi19} for the case with $\nu<1$ and $\rho=2$. 
For a smooth convex $f$ with $\nu=1$ and $g$ being the indicator function of a uniformly convex set, similar iteration complexity bounds as in (iv) and (v) with $\nu=1$ were established  in \cite[]{Ker21a} for a parameter-dependent conditional gradient method for reaching 
the criterion $\varphi(x_t)-\varphi^* \le \varepsilon$.

In addition, as observed from Theorem~\ref{lem:L-est} (iii), the total number of inner iterations of Algorithm~\ref{alg:main} for reaching the termination criterion $\delta_t \le \ep$ is at most $2t + [\log_2 (2\overline{L}(\ep) /L_{-1})]_+$. Also, notice from \eqref{eq:def-tilde-L} that $\log\overline{L}(\ep)=O(\log(1/\ep))$. In view of these, one can see that the total number of inner iterations of Algorithm~\ref{alg:main} enjoys the same complexity bounds as given in (i)-(v) for reaching the termination criterion $\delta_t \le \ep$.


\section{Numerical Experiments} \label{sec:num}

In this section we conduct some numerical experiments to compare the performance of the conditional gradient methods studied in this paper with the ones with some other step size rules. For the comparison, we construct the problems whose H\"older continuity exponent $\nu$ and uniform convexity exponent $\rho$ are known in advance. More specifically, we generate the test instances from the problem classes discussed in Examples~\ref{ex:uniconv-set} and \ref{ex:uniconv-func}, respectively. Our experiments are conducted in Matlab on an Apple desktop with the 3.0GHz Intel Xeon E5-1680v2 processor and 64GB of RAM.


\subsection{$\ell_p$-Norm Minimization over $\ell_q$ Ball} \label{lq-ball}
In our first experiment, we consider the following problem:
\begin{equation}\label{eq:ex-lp}
\begin{array}{rl}
\min & \frac{1}{p}\norm{Ax-b}_p^p\\
\text{s.t.} & x \in \cB_q := \{z \in \RR^n: \norm{z}_q \leq 1\},
\end{array}
\end{equation}
where $1<p \leq 2$, $q >1$, $A \in \RR^{m\times n}$, and $b \in \RR^m$. Note that $\cB_q$ is a uniformly convex set with exponent $\rho=\max(2,q)$ \cite[e.g., see][]{Ker21a,Ker21}. In addition, as shown in Lemma~\ref{lem:lp-Holder} in Section~\ref{sec:aux-lem}, the gradient of the objective function of 
\eqref{eq:ex-lp} is H\"older continuous with respect to $\|\cdot\|_2$ with exponent $\nu=p-1$ and modulus $M_{\nu}=2^{2-p} m^{\frac{(p-1)(2-p)}{2p}} \norm{A}_2^p$. Thus, problem \eqref{eq:ex-lp} belongs to the class of the problems minimizing a weakly smooth convex function over a uniformly convex set discussed in Example~\ref{ex:uniconv-set}. Note that when $A=I$, it reduces to the problem of the 
$\ell_p$-norm projection of a vector onto the $\ell_q$ unit ball.

We next apply the following three conditional gradient methods 
to solve problem \eqref{eq:ex-lp}, and compare their performance.
\begin{itemize}
\item Algorithm~\ref{alg:FW-exact-ls} with $\|\cdot\|=\|\cdot\|_2$, $\nu=p-1$, and $M_{\nu}=2^{2-p} m^{\frac{(p-1)(2-p)}{2p}} \norm{A}_2^p$.
\item Algorithm~\ref{alg:main} with $\|\cdot\|=\|\cdot\|_2$.
\item The conditional gradient method with the well-known diminishing step size $\tau_t = 2/(t+2)$ \cite[]{Jaggi11,FG16}, which is similar to Algorithm~\ref{alg:FW-exact-ls} except the choice of $\tau_t$.
\end{itemize}
As discussed in Section~\ref{sec:iter-compl}, for finding an approximate solution of \eqref{eq:ex-lp} with an $\ep$-Frank-Wolfe gap, the conditional gradient method with step size $\tau_t = 2/(t+2)$ enjoys an iteration complexity of $O(\ep^{-1/\nu})$, while Algorithms~\ref{alg:FW-exact-ls} and \ref{alg:main} enjoy the following iteration complexity:  
$$
\left\{
\begin{array}{ll}
O(\log(1/\ep))  & \text{if } p=2 \text{ and } q\leq 2,\\
O(\ep^{-(\rho-1-\nu)/(\rho\nu)})) & \text{otherwise},
\end{array}
\right.
\quad
\text{with } \rho=\max(2,q) \text{ and } \nu=p-1.
$$

When applied to \eqref{eq:ex-lp}, the above three  methods need to solve the subproblems of the form 
\[
\min\limits_x \{\innprod{u}{x}:\norm{x}_q \leq 1\}
\]
for $u \in \RR^n$. It is not hard to observe that this problem has a closed-form solution given by
$$
x_i^* = -\norm{u}_q^{-\frac{1}{q-1}}\sign(u_i)|u_i|^{\frac{1}{q-1}},\quad i=1,\ldots,n.
$$

 The instances of problem \eqref{eq:ex-lp} are generated as follows. In particular, we generate matrix $A$ by letting $A=UDU^T$, where $D \in \RR^{n\times n}$ is a diagonal matrix, whose diagonal entries are randomly generated according to  the uniform distribution over $[1,100]$ and $U\in\RR^{n\times n}$ is a randomly generated orthogonal matrix. We set $b=A\bar{x}$ for some $\bar{x}$ generated from a uniform distribution over $\{x\in\RR^n:\norm{x}_q=10\}$. 
 
In this experiment, we consider $p \in \{1.3,\,1.6,\,2\}$, $q \in \{1.5,\,2,\,3\}$ and $m=n \in \{1000,5000\}$. For each choice of $(p,q,n)$, we randomly generate $10$ instances of problem \eqref{eq:ex-lp} by the procedure mentioned above, and apply the aforementioned three conditional gradient methods to solve them, starting with the initial point $x_0=0$ and terminating them once the criterion $\delta_t/\delta_0 \leq 10^{-6}$ is met, where $\delta_t$ and $\delta_0$ are the Frank-Wolfe gap at the iterates $x_t$ and $x_0$, respectively. Table~\ref{table:num1} presents the average CPU time (in seconds) and the average number of iterations of these methods over the $10$ random instances. In detail, the values of $n,q,p$ are given in the first three columns, and the average CPU time and the average number of iterations of Algorithms 1, 2 and the conditional gradient method with step size $\tau_t=2/(t+2)$ are given in the rest of the columns. In addition, Figure~\ref{fig:num1} illustrates the behavior of the best relative Frank-Wolfe gap $\delta_t^*/\delta_0 := \min_{0\leq i \leq t}\delta_i/\delta_0$ 
and the objective value gap $\varphi(x_t)-\widetilde{\varphi}_*$ 
with respect to CPU time on a single random instance of problem \eqref{eq:ex-lp} with $n=5000$, $q=3$, and $p=1.3$, $1.6$, $2$, respectively, 
where $\widetilde{\varphi}_*$ is the minimum objective function value of all iterates generated by the three algorithms. One can see that Algorithm~\ref{alg:main} generally outperforms the other two methods. This is perhaps because: (i) Algorithm~\ref{alg:main}  improves the iteration complexity of the conditional gradient method with step size $\tau_t=2/(t+2)$; (ii) Algorithm~\ref{alg:main} uses an adaptive step size determined by using a constructive local quadratic upper approximation of the objective function and an adaptive line search scheme.

\begin{table}[htbp]
\centering
\begin{tabular}{ccc||lll||lll}
\hline
&  & & \multicolumn{3}{c||}{Average CPU time (sec)} & \multicolumn{3}{c}{Average number of iterations}\\
$n$ & $q$ & $p$ & Algorithm~\ref{alg:FW-exact-ls} & Algorithm~\ref{alg:main} & $\frac{2}{t+2}$  & Algorithm~\ref{alg:FW-exact-ls} & Algorithm~\ref{alg:main} &  $\frac{2}{t+2}$ \\ \hline\hline
1000& 1.5 & 1.3  &  3.51 & 0.054 & 0.56 & 8881.4 & 84.9 & 1404.5\\
 & & 1.6  &  0.0065 & 0.0055 & 0.52 & 13.5 & 6.2 & 1333.5\\
 & & 2.0  &  0.0028 & 0.0037 & 0.44 & 5.0 & 6.2 & 1287.1\\
 & 2.0 & 1.3  &  1.63 & 0.13 & 0.50 & 4901.2 & 252.5 & 1544.4\\
 & & 1.6  &  0.0054 & 0.0060 & 0.44 & 13.8 & 6.9 & 1335.1\\
 & & 2.0  &  0.0020 & 0.0022 & 0.37 & 4.0 & 4.0 & 1299.4\\
 & 3.0 & 1.3  &  10.9 & 1.18 & 1.76 & 27442.3 & 2038.1 & 4449.2\\
 & & 1.6  &  0.028 & 0.012 & 0.52 & 68.4 & 18.5 & 1323.1\\
 & & 2.0  &  0.0036 & 0.0038 & 0.45 & 7.7 & 7.4 & 1289.8\\
 \hline
5000& 1.5 & 1.3  &  20.6 & 1.34 & 13.3 & 2223.8 & 132.5 & 1424.4\\
 & & 1.6  &  0.063 & 0.078 & 12.5 & 5.7 & 6.2 & 1334.6\\
 & & 2.0  &  0.056 & 0.067 & 12.0 & 5.0 & 6.2 & 1288.0\\
 & 2.0 & 1.3  &  7.03 & 3.23 & 13.2 & 809.5 & 341.3 & 1506.4\\
 & & 1.6  &  0.10 & 0.083 & 11.6 & 10.5 & 7.2 & 1335.8\\
 & & 2.0  &  0.044 & 0.043 & 11.2 & 4.0 & 4.0 & 1300.1\\
 & 3.0 & 1.3  &  161.2 & 28.4 & 37.3 & 17364.6 & 2827.7 & 3972.8\\
 & & 1.6  &  0.41 & 0.20 & 12.4 & 43.7 & 18.7 & 1323.4\\
 & & 2.0  &  0.084 & 0.083 & 12.0 & 7.8 & 7.8 & 1289.9\\
    \hline
\end{tabular}
\caption{Numerical results for problem \eqref{eq:ex-lp}}
\label{table:num1}
\end{table}

\begin{figure}[htbp]
 \begin{minipage}{0.32\hsize}
  \begin{center}
   \includegraphics[width=50mm]{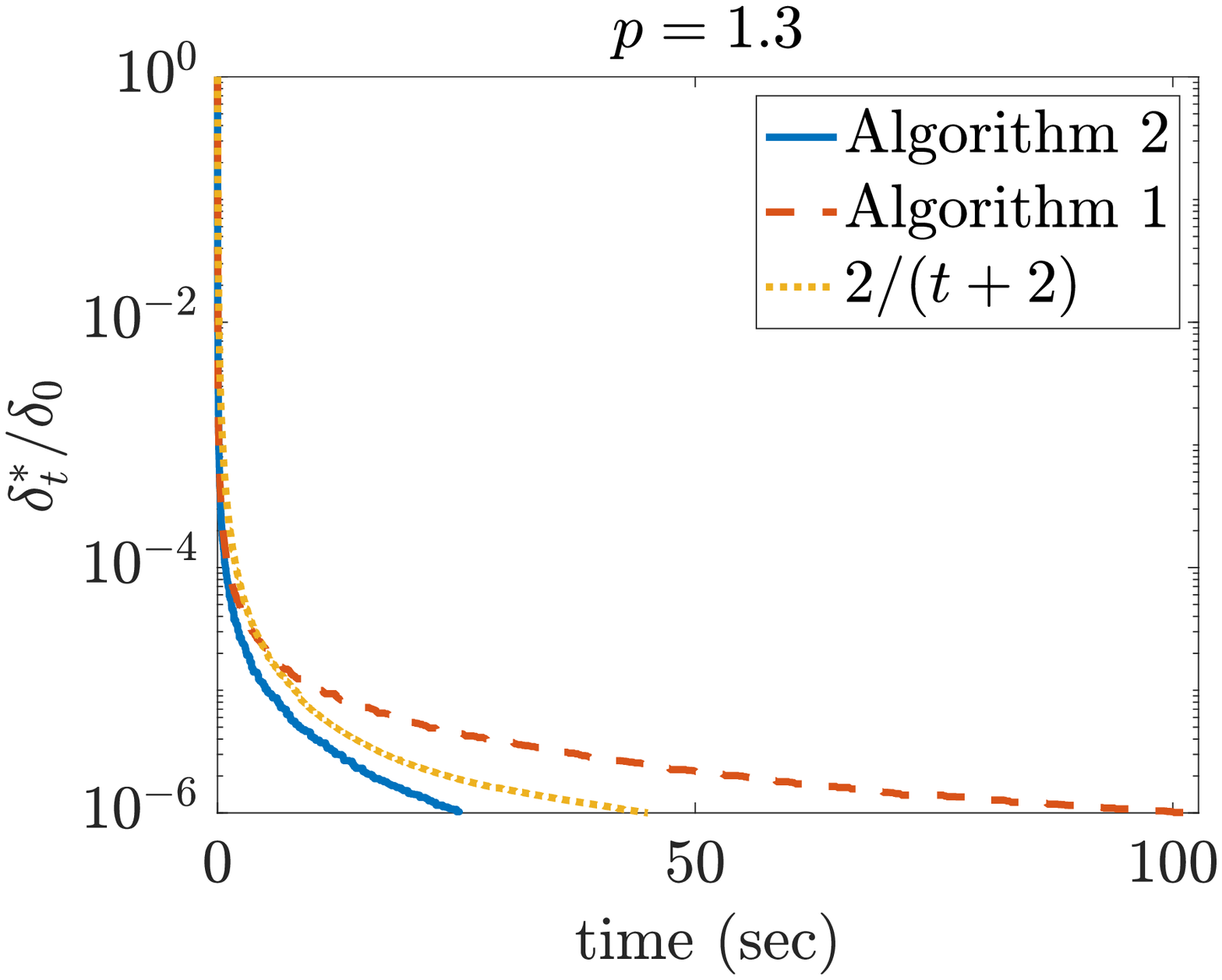}
  \end{center}
 \end{minipage}
 \begin{minipage}{0.32\hsize}
 \begin{center}
  \includegraphics[width=50mm]{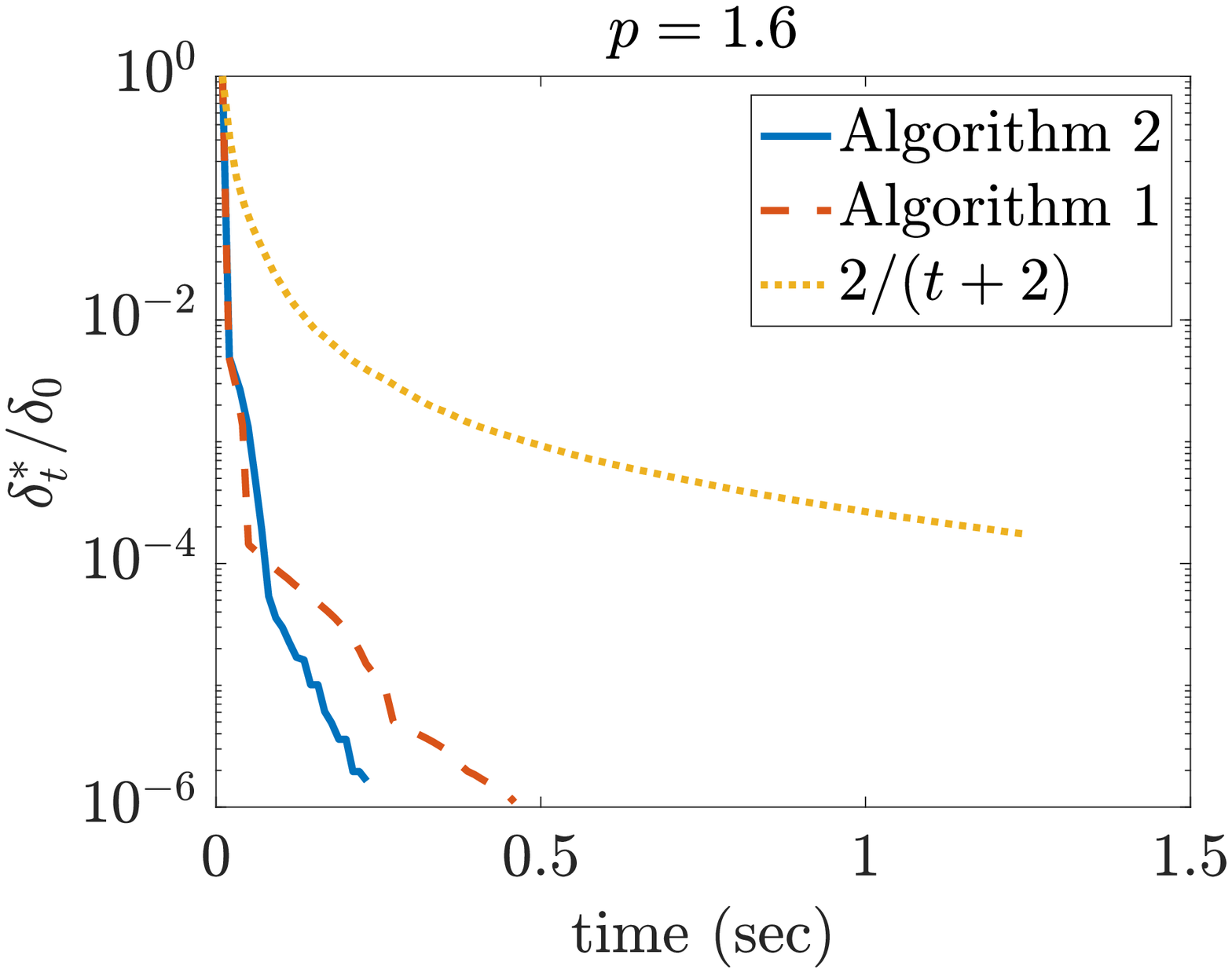}
 \end{center}
 \end{minipage}
 \begin{minipage}{0.32\hsize}
 \begin{center}
  \includegraphics[width=50mm]{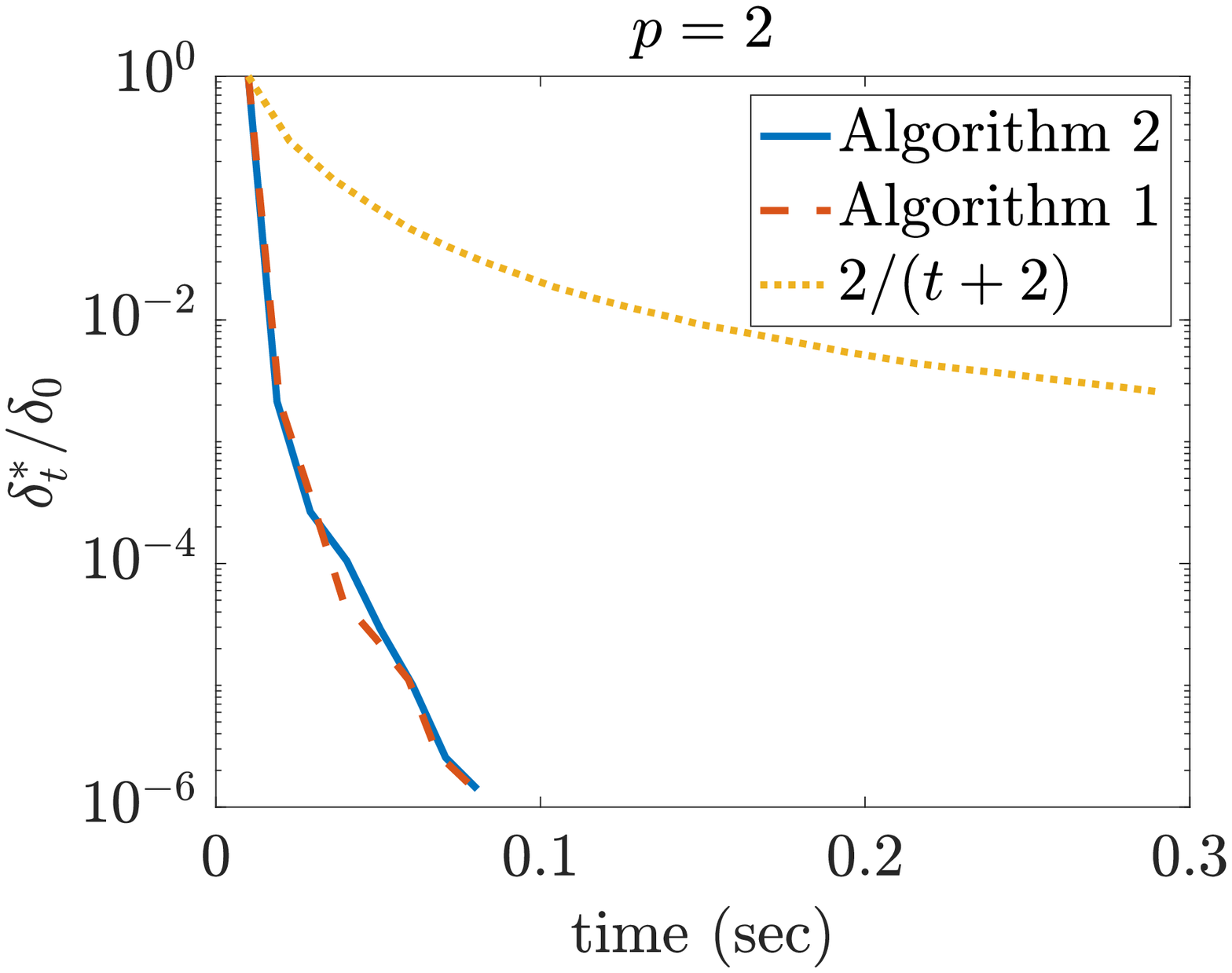}
 \end{center}
 \end{minipage}
 \\[2mm]
 \begin{minipage}{0.32\hsize}
  \begin{center}
   \includegraphics[width=50mm]{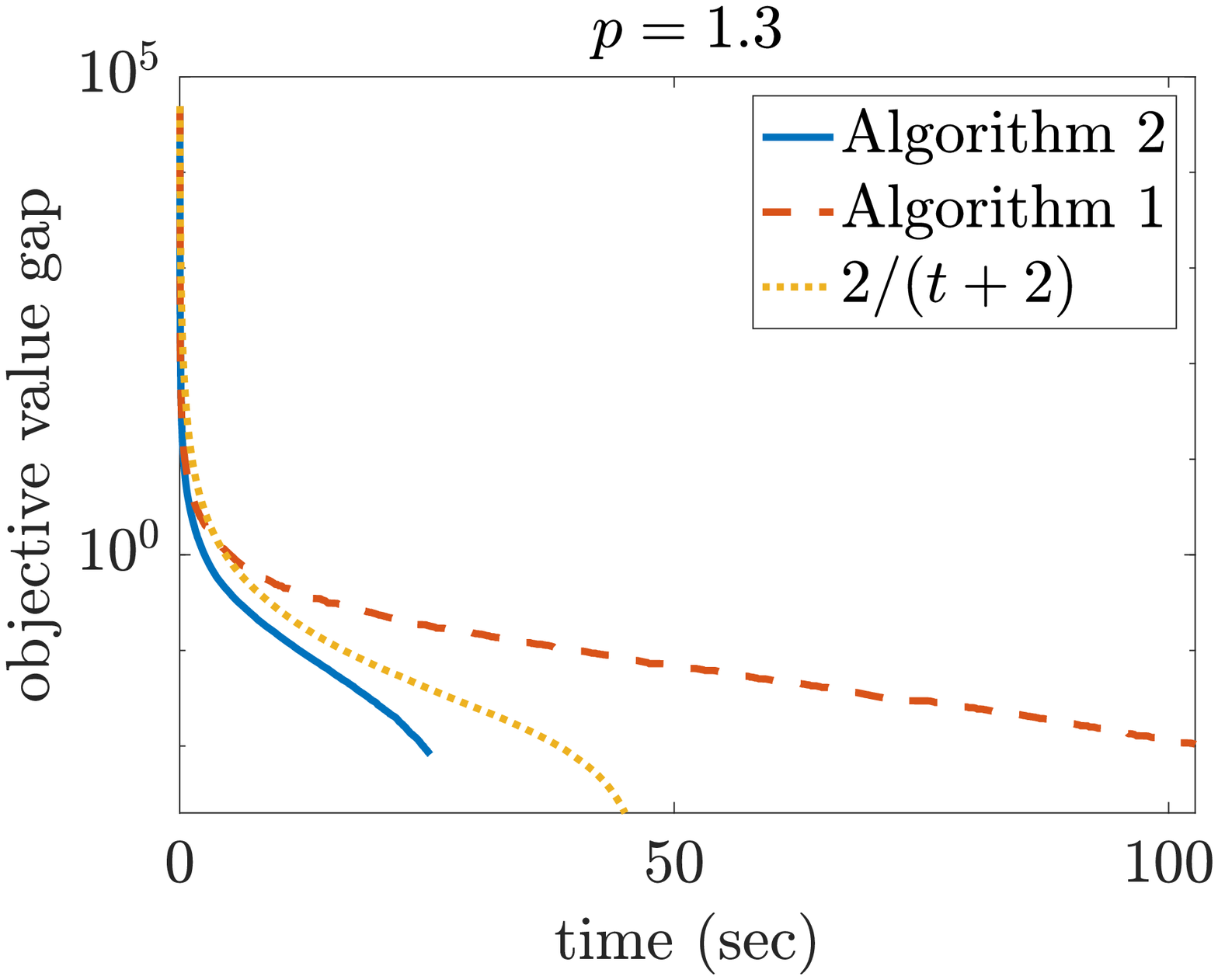}
  \end{center}
 \end{minipage}
 \begin{minipage}{0.32\hsize}
 \begin{center}
  \includegraphics[width=50mm]{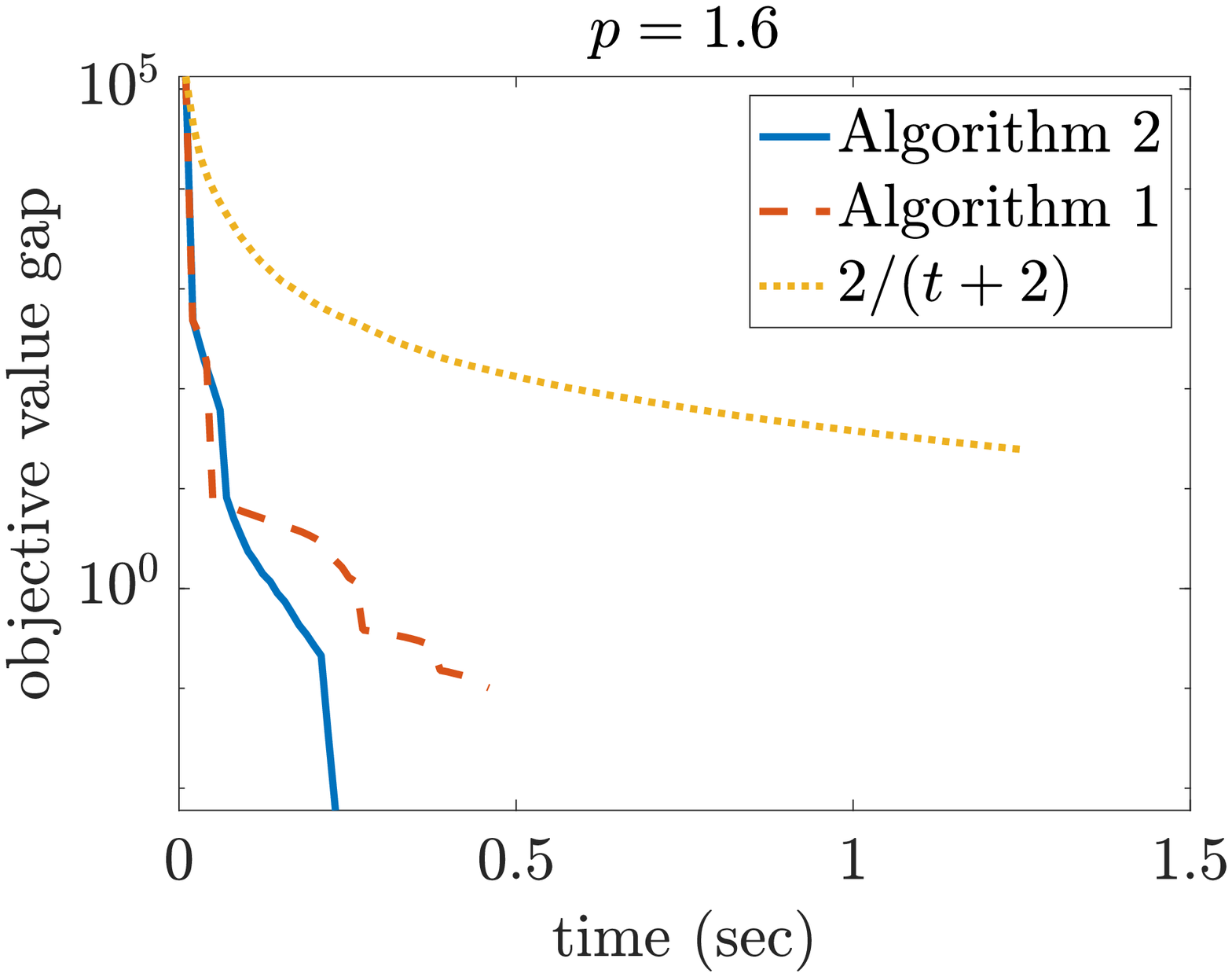}
 \end{center}
 \end{minipage}
 \begin{minipage}{0.32\hsize}
 \begin{center}
  \includegraphics[width=50mm]{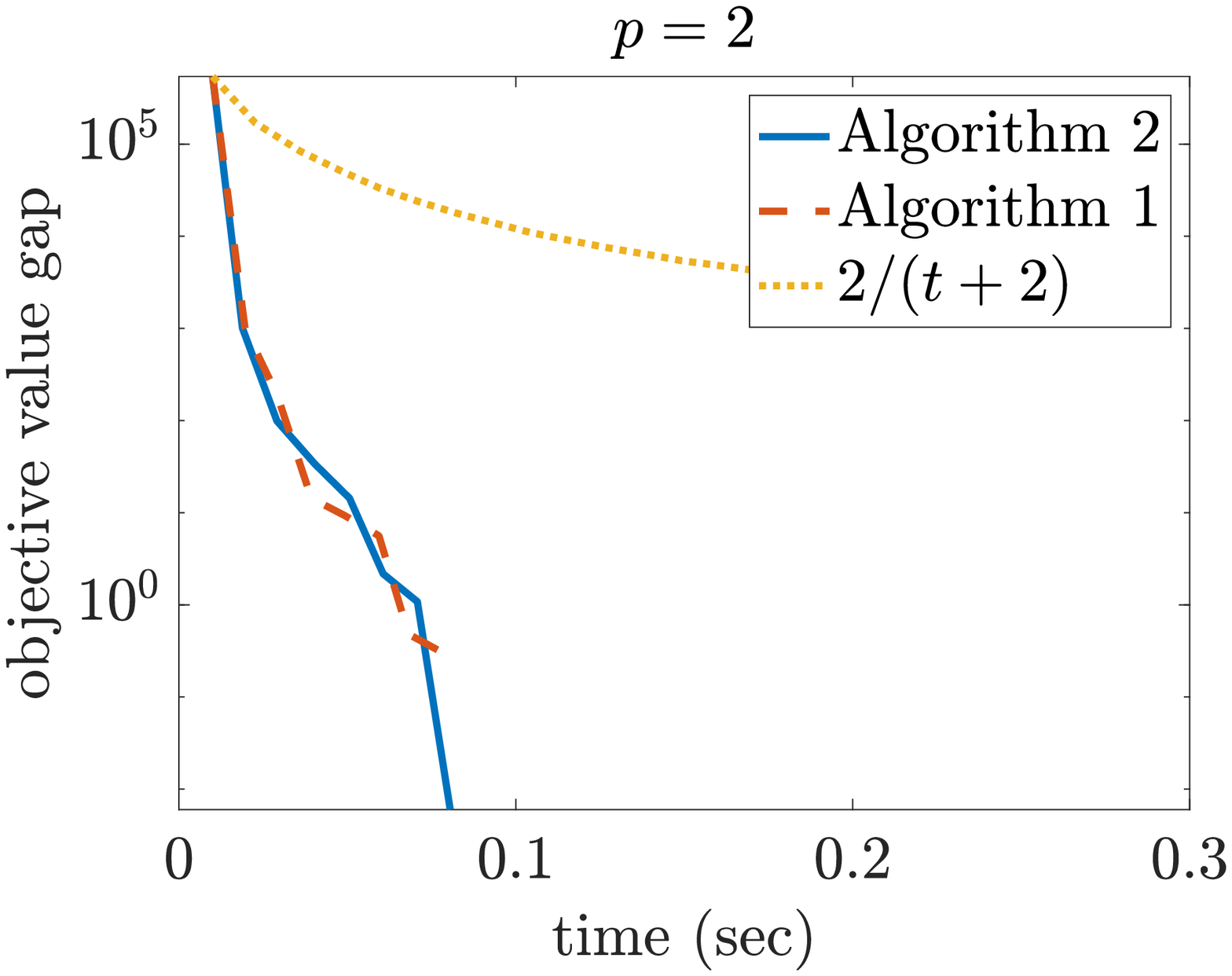}
 \end{center}
 \end{minipage}
 \caption{Numerical results on a single random instance of problem \eqref{eq:ex-lp} with $n=5000$, $q=3$, and $p=1.3$, $1.6$, $2$, respectively. These sub-figures illustrate the behavior of the best relative Frank-Wolfe gap $\delta_t^*/\delta_0 := \min_{0\leq i \leq t}\delta_i/\delta_0$ and the objective value gap $\varphi(x_t)-\widetilde{\varphi}_*$ with respect to CPU time in seconds, where $\widetilde{\varphi}_*$ is the minimum objective function value of all iterates generated by the three algorithms for solving one problem instance.}
 \label{fig:num1}
\end{figure}


\subsection{Entropy Regularized $\ell_p$-Norm Minimization}
In our second experiment, we consider the following problem:
\begin{equation}\label{eq:ex-entropy}
\begin{array}{ll}
\text{min} & \frac{1}{p}\norm{Ax-b}_p^p+\lambda \sum_{i=1}^n x_i\log x_i\\
\text{s.t.} & x \in \Delta_n := \{z \in \RR^n: \sum^n_{i=1}z_i=1,~z_i\geq 0,~i=1,\ldots,n\},
\end{array}
\end{equation}
where $p >1$, $\lambda>0$, $A \in \RR^{m\times n}$, and $b \in \RR^m$.
Note that $f(x)=\frac{1}{p}\norm{Ax-b}_p^p$ is weakly smooth and $g(x)=\lambda \sum_{i=1}^n x_i\log x_i + \iota_{\Delta_n}(x)$ is $\lambda$-strongly convex with respect to the $\ell_1$-norm \cite[e.g., see][]{BT03}, where $\iota_{\Delta_n}$ denotes the indicator function of $\Delta_n$. Thus, problem \eqref{eq:ex-entropy} is a special case of Example~\ref{ex:uniconv-func} with $\nu=p-1$ and $\rho=2$.

We next apply Algorithms~\ref{alg:FW-exact-ls} and~\ref{alg:main} with the same settings as in Subsection~\ref{lq-ball} and the conditional gradient method with the diminishing step size 
\begin{equation}\label{eq:dms2}
\tau_t = \frac{6(t+1)}{(t+2)(2t+3)}
\end{equation}
proposed by \cite{Nes18} to solve problem \eqref{eq:ex-entropy}, and compare their performance. The latter method is similar to Algorithm~\ref{alg:FW-exact-ls} except the choice of $\tau_t$ and enjoys an iteration complexity of 
$O(\ep^{-1/(2(p-1))})$ for finding an approximate solution of \eqref{eq:ex-entropy} with an $\ep$-Frank-Wolfe gap \cite[]{Nes18}. In addition, as seen from Section~\ref{sec:iter-compl}, Algorithms~\ref{alg:FW-exact-ls} and \ref{alg:main} enjoy the following iteration complexity for finding an approximate solution of \eqref{eq:ex-entropy} with an $\ep$-Frank-Wolfe gap:
$$
\left\{
\begin{array}{ll}
O\left(\frac{\norm{A}_2^2}{\lambda}\log\frac{1}{\ep}\right)  & \text{if } p=2,\\
O(\ep^{-(2-p)/(2(p-1))})) & \text{otherwise}.
\end{array}
\right.
$$

When applied to \eqref{eq:ex-entropy}, the aforementioned three methods need to solve the subproblems of the form 
\[
\min_{x \in \Delta_n} \innprod{u}{x}+\lambda \sum_{i=1}^n x_i\log x_i
\]
for some $u\in\RR^n$. It is well-known that this problem has a closed-form solution given by
\[
x^*_i = \frac{e^{-u_i/\lambda}}{\sum_{j=1}^n e^{-u_j/\lambda}},\quad i=1,\ldots,n.
\]

 The instances of problem \eqref{eq:ex-entropy} are generated as follows. In particular,  we generate matrix $A$ with  $\norm{A}_2 \leq 100$ by letting $A=VDU^T$, where $D$ is a $m\times m$ diagonal matrix, whose diagonal entries are randomly generated according to the uniform distribution over $[0,100]$, and $U\in\RR^{n\times m}$ and $V \in \RR^{m\times m}$ are randomly generated orthonormal matrices. Each entry of $b \in \RR^m$ is generated from the uniform distribution on $[0,1]$. 
 
In this experiment, we consider $m = n/2 \in \{1000,5000\}$, $p \in \{1.5, 1.75, 2\}$, and $\lambda \in \{1,10,50\}$. For each choice of $(m,n,p,\lambda)$, we randomly generate $10$ instances of problem \eqref{eq:ex-entropy} by the procedure mentioned above, and apply the aforementioned three conditional gradient methods to solve them, starting with the initial point $x_0=(1/n,\ldots,1/n)^T$ and terminating them once the criterion $\delta_t/\delta_0 \leq 10^{-8}$ is met, where $\delta_t$ and $\delta_0$ are the Frank-Wolfe gap at the iterates $x_t$ and $x_0$, respectively. Table~\ref{table:num2} presents the average CPU time (in seconds) and the average number of iterations of these methods over the $10$ random instances. 
In detail, the values of $m,p,\lambda$ are given in the first three columns, and the average CPU time and the average number of iterations of Algorithms 1, 2 and the conditional gradient method with step size $\tau_t=6(t+1)/((t+2)(2t+3))$ are given in the rest of the columns. In addition, Figure~\ref{fig:num2} illustrates the behavior of the best relative Frank-Wolfe gap $\delta_t^*/\delta_0 := \min_{0\leq i \leq t}\delta_i/\delta_0$ 
and the objective value gap
$\varphi(x_t)-\widetilde{\varphi}_*$ 
 with respect to CPU time on a single random instance of problem \eqref{eq:ex-entropy} with $m=5000$, $n=$\,10,000, $\lambda=10$, and $p=1.5, 1.75, 2$, respectively, where $\widetilde{\varphi}_*$ is the minimum objective function value of all iterates generated by the three algorithms. One can see that Algorithm~\ref{alg:main} generally outperforms the other two methods, which is perhaps for the similar reasons as explained at the end of Subsection~\ref{lq-ball}.

\begin{table}[p]
\centering
\resizebox{\textwidth}{!}{
\begin{tabular}{ccc||lll||lll}
\hline
&  & & \multicolumn{3}{c||}{Average CPU time (sec)} & \multicolumn{3}{c}{Average number of iterations}\\
$m$ & $p$ & $\lambda$ & Algorithm~\ref{alg:FW-exact-ls} & Algorithm~\ref{alg:main} & Rule~\eqref{eq:dms2} & Algorithm~\ref{alg:FW-exact-ls} & Algorithm~\ref{alg:main} & Rule~\eqref{eq:dms2} \\ \hline\hline
1000& 1.5 & 1  &  314.3 & 0.47 & 0.49 & 640908.6 & 607.5 & 982.4\\
& & 10  &  33.4 & 0.030 & 0.046 & 67184.7 & 34.1 & 87.4\\
& & 50  &  2.60 & 0.010 & 0.016 & 5186.6 & 8.1 & 28.1\\
& 1.75 & 1  &  0.77 & 0.17 & 0.39 & 1446.4 & 210.8 & 731.3\\
& & 10  &  0.21 & 0.020 & 0.048 & 404.3 & 21.4 & 85.0\\
& & 50  &  0.015 & 0.0056 & 0.016 & 27.6 & 5.0 & 28.0\\
& 2 & 1  &  0.12 & 0.14 & 0.32 & 236.8 & 237.1 & 647.0\\
& & 10  &  0.021 & 0.015 & 0.050 & 40.2 & 24.0 & 94.7\\
& & 50  &  0.0035 & 0.0047 & 0.015 & 5.0 & 5.1 & 27.4\\
\hline
5000& 1.5 & 1  &  8300.7 & 5.23 & 8.99 & 459440.4 & 272.6 & 518.4\\
& & 10  &  1943.9 & 0.33& 1.77 & 107117.5 & 16.3 & 100.4\\
& & 50  &  16.5 & 0.10 & 0.55 & 952.0 & 4.3 & 31.0\\
& 1.75 & 1  &  30.9 & 2.77 & 8.99 & 1716.2 & 143.0 & 513.6\\
& & 10  &  6.28 & 0.30 & 1.86 & 353.9 & 14.7 & 106.3\\
& & 50  &  0.16 & 0.071 & 0.56 & 8.3 & 3.0 & 30.6\\
& 2 & 1  &  4.00 & 2.67 & 8.99 & 230.7 & 150.7 & 513.7\\
& & 10  &  0.78 & 0.33 & 2.11 & 44.8 & 17.5 & 119.2\\
& & 50  &  0.071 & 0.070 & 0.55 & 3.0 & 3.0 & 31.0\\
    \hline
\end{tabular}
}
\caption{Numerical results for problem \eqref{eq:ex-entropy}}
\label{table:num2}
\end{table}

\begin{figure}[p]
 \begin{minipage}{0.32\hsize}
  \begin{center}
   \includegraphics[width=50mm]{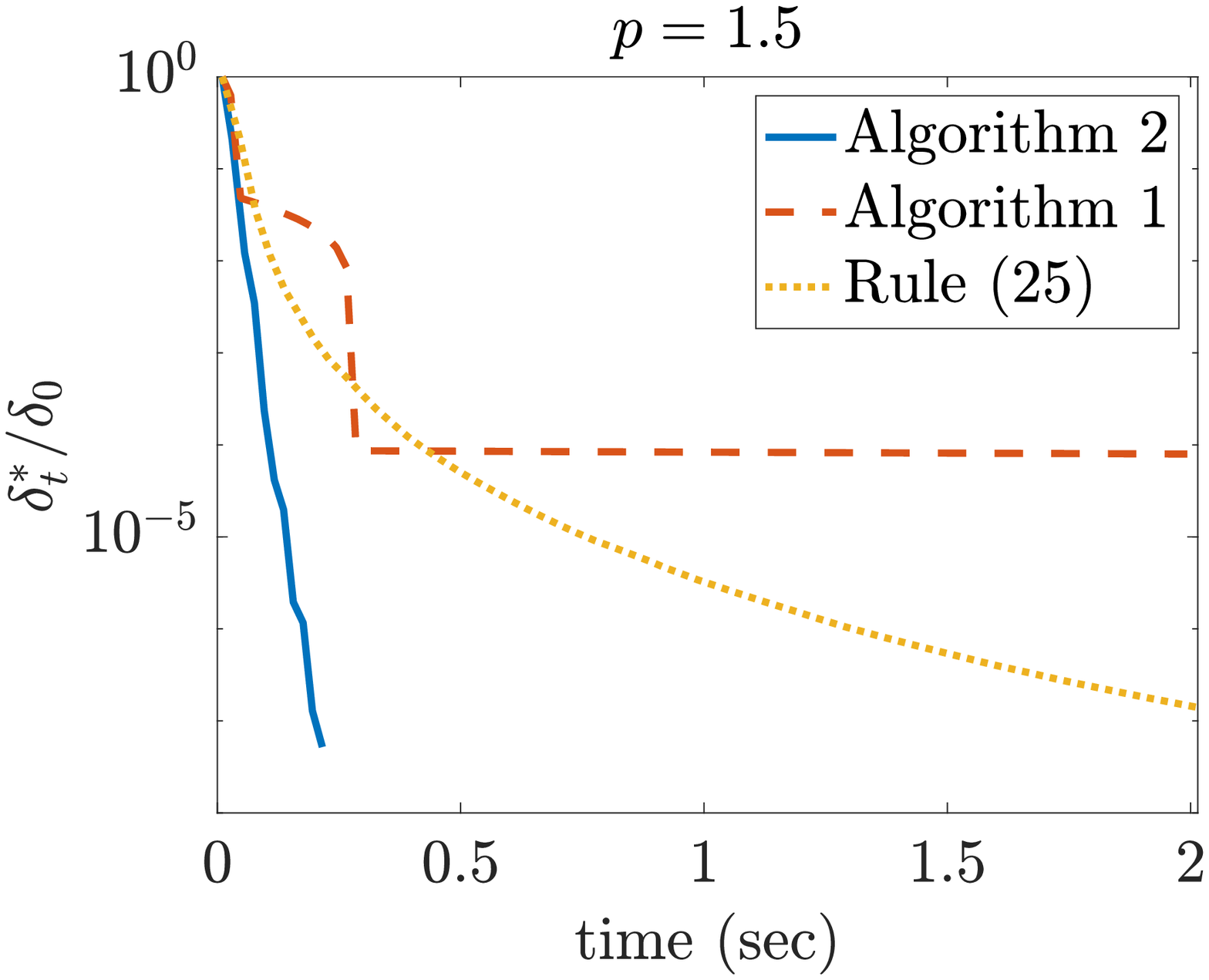}
  \end{center}
 \end{minipage}
 \begin{minipage}{0.32\hsize}
 \begin{center}
  \includegraphics[width=50mm]{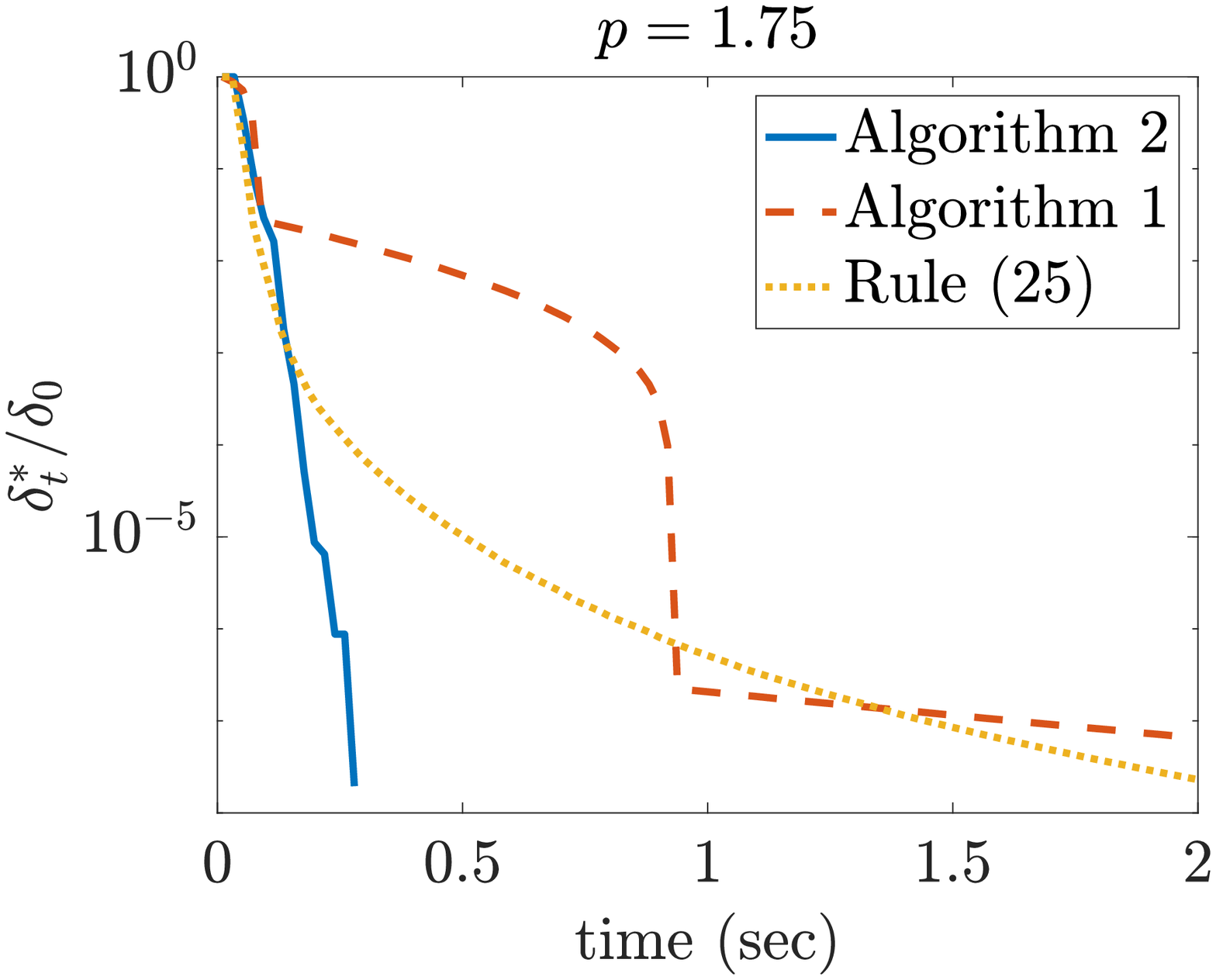}
 \end{center}
 \end{minipage}
 \begin{minipage}{0.32\hsize}
 \begin{center}
  \includegraphics[width=50mm]{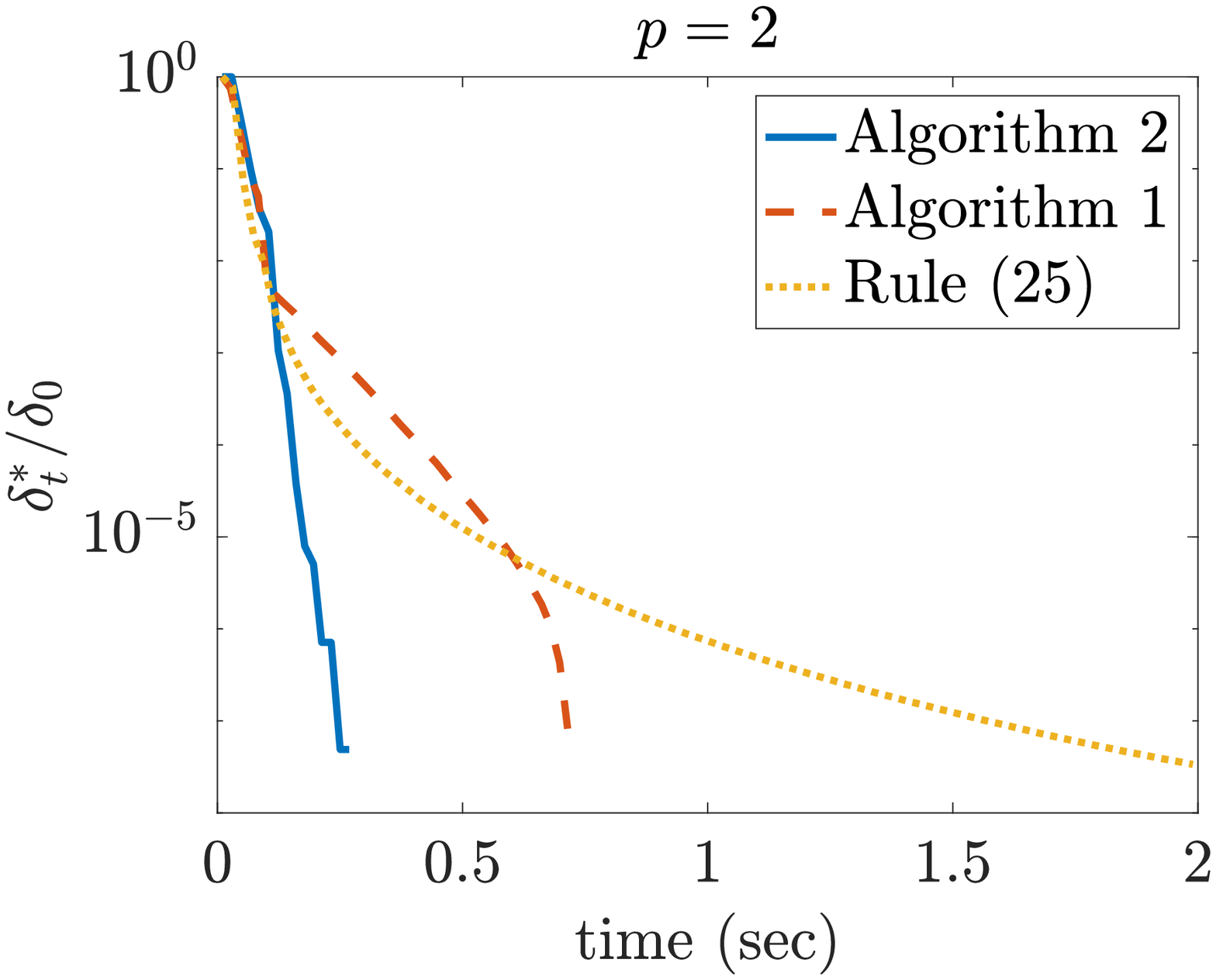}
 \end{center}
 \end{minipage}
  \\[2mm]
 \begin{minipage}{0.32\hsize}
  \begin{center}
   \includegraphics[width=50mm]{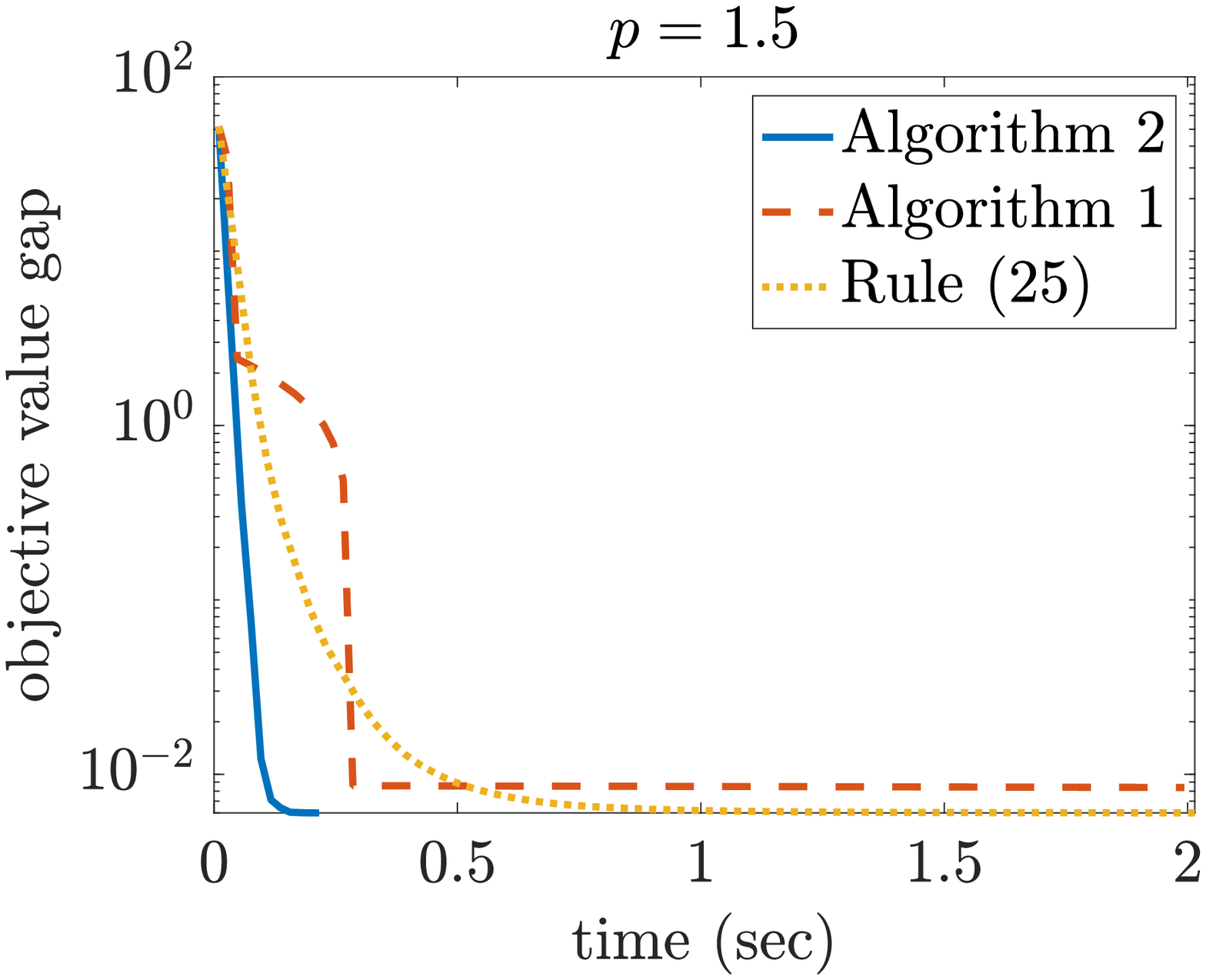}
  \end{center}
 \end{minipage}
 \begin{minipage}{0.32\hsize}
 \begin{center}
  \includegraphics[width=50mm]{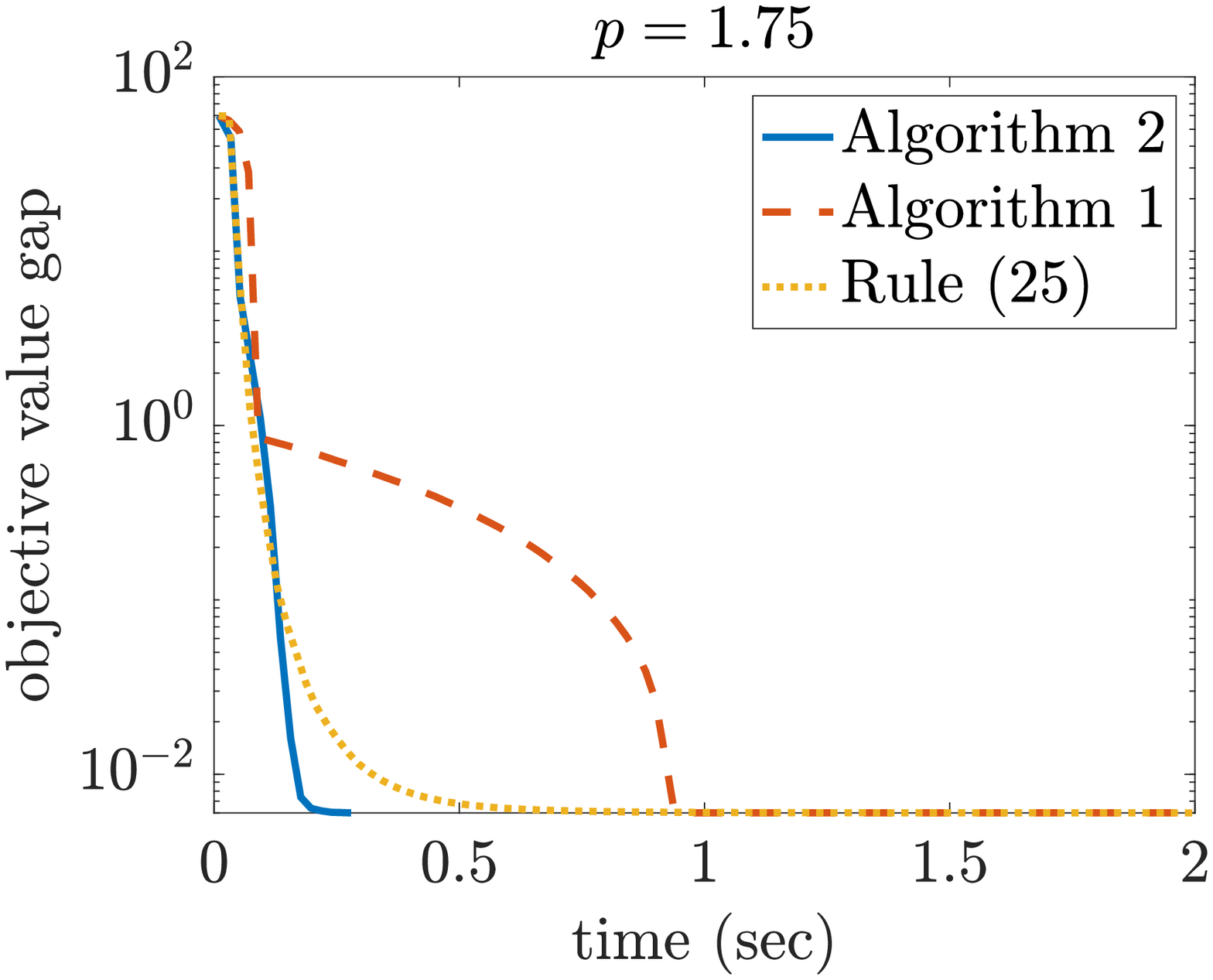}
 \end{center}
 \end{minipage}
 \begin{minipage}{0.32\hsize}
 \begin{center}
  \includegraphics[width=50mm]{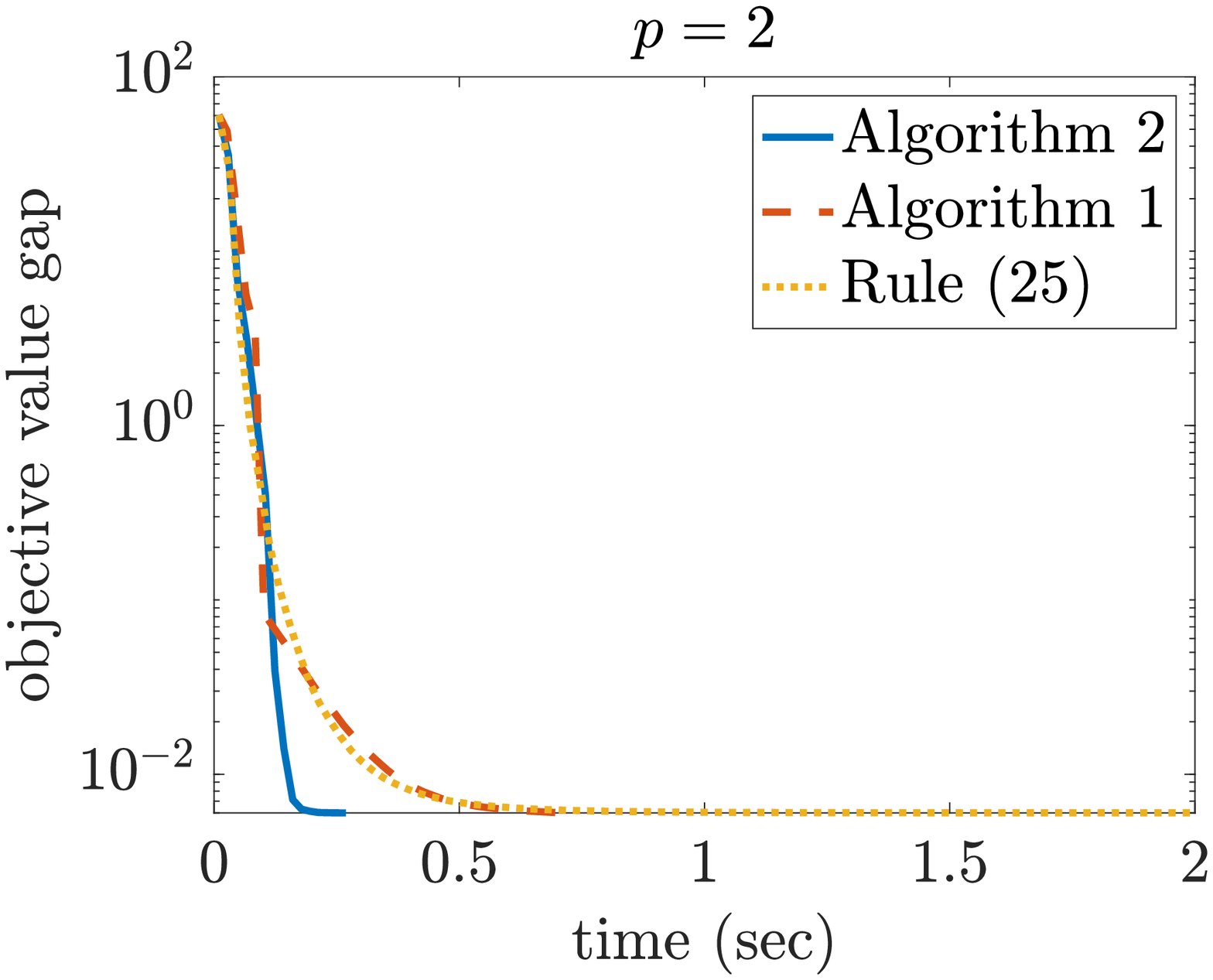}
 \end{center}
 \end{minipage}
 \caption{Numerical results on a single random instance  of problem \eqref{eq:ex-entropy} with $m=5000$, $n=$\,10,000,  $\lambda=10$, and $p=1.5, 1.75, 2$, respectively. These sub-figures illustrate the behavior of the best relative Frank-Wolfe gap $\delta_t^*/\delta_0 := \min_{0\leq i \leq t}\delta_i/\delta_0$ and the objective function value gap $\varphi(x_t)-\widetilde{\varphi}_*$ with respect to CPU time in seconds. Here, $\widetilde{\varphi}_*$ denotes the minimum objective function value of all iterates generated by the three algorithms for solving one problem instance.}
 \label{fig:num2}
\end{figure}

\subsection{Simplex-Constrained Nonnegative Matrix Factorization}

In our third experiment, we consider the following simplex-constrained nonnegative matrix factorization problem \cite[see][]{OlNiFa22nmf}:
\begin{equation}\label{eq:nmf}
\begin{array}{rl}
\min & \frac{1}{2}\|X-UV\|_F^2 + \lambda(\|U\|_F^2+\|V\|_F^2)\\[1pt]
\text{s.t.} & U\in B_{n, k}:=\{U\in\RR^{n\times k} ~|~ 0\leq U_{ij}\le\alpha,\ 1\le i\le n,\, 1\le j\le k\},\\[1pt]
&V\in \Delta_{k, m}:=\{V\in\RR^{k\times m}_+ ~|~ V^T1_k = 1_m\},
\end{array}
\end{equation}
where $X \in \RR^{n \times m}$, $\alpha,\lambda>0$,
 $\|\cdot\|_F$ is the Frobenius norm, and $1_d\in\RR^d$ is the all-ones vector for any $d\ge 1$. Problem \eqref{eq:nmf} can be viewed as $\min_{U,V}\{\varphi(U,V):=f(U,V)+g(U,V)\}$ with
$$
f(U,V)=\frac{1}{2}\|X-UV\|_F^2,\quad g(U,V)=\lambda(\|U\|_F^2 + \|V\|_F^2) + \iota_{B_{n,k}}(U)+\iota_{\Delta_{k,m}}(V),
$$
where $\iota_{B_{n,k}}$ and $\iota_{\Delta_{k,m}}$ denote the indicator function of $B_{n,k}$ and $\Delta_{k,m}$, respectively. 
Notice that $f$ is nonconvex and smooth, $\dom{g}$ is compact, and $g$ is strongly convex.  Clearly, problem~\eqref{eq:nmf} is a special case of problem \eqref{main-prob} satisfying Assumptions \ref{assump:fg} and \ref{assump:g} with $\nu=1$ and $\rho=2$. 

We next apply the following two conditional gradient methods 
to solve problem~\eqref{eq:nmf}, and compare their performance.
\begin{itemize}
\item Algorithm~\ref{alg:main} with $\|\cdot\|=\|\cdot\|_F$.
\item The conditional gradient method with line search \cite[Algorithm~2]{Ghadimi19}, abbreviated as CGM-LS.
We set the parameters $\gamma=0.5$ and $\delta=\ep \delta_0/4$ for this method as suggested in \cite[]{Ghadimi19}, where $\delta_0$ is the Frank-Wolfe gap at the initial point and $\ep$ is the targeted tolerance for the final relative Frank-Wolfe gap. 
\end{itemize}
It shall be mentioned that these two methods enjoy an iteration complexity of $O(1/\varepsilon)$ for finding an approximate solution of \eqref{eq:nmf} with an $\ep$-Frank-Wolfe gap (see Section~\ref{sec:iter-compl} and equation (1.8) in \cite[]{Ghadimi19}).

The data matrix $X \in \RR^{n\times m}$ for problem~\eqref{eq:nmf} is generated as follows. In particular, we first randomly generate $U^* \in \RR^{n\times k}$ with all entries following the uniform distribution over $[0,\alpha]$. We next randomly generate $\widetilde{V} \in \RR^{k\times m}$ with all entries following the standard normal distribution and set $V^*=\widetilde{V}D$, where $D \in \RR^{m\times m}$ is a diagonal matrix such that $(V^*)^T1_k=1_m$. Finally, we set $X=U^*V^*+E$, where the entries of $E \in \RR^{n\times m}$ follow the normal distribution with mean zero and standard deviation $0.01$.

In this experiment, we set $\lambda=0.01$, $\alpha=2$, and
consider $m=n\in\{100,200,300,400,500\}$ and $k\in\{5,10\}$. For each choice of $(m,n,k)$, we randomly generate $10$ instances of problem~\eqref{eq:nmf} by the procedure mentioned above. Then we apply the aforementioned two conditional gradient methods to solve them with the initial point $U^0$ and $V^0$ being the  matrices of all entries equal to $1$ and $1/k$, respectively, and terminate the methods once the criterion $\delta_t/\delta_0\le 10^{-5}$ is met, where $\delta_t$ is the Frank-Wolfe gap at the $t$-th iteration $(U^t,V^t)$. The computational results are presented in Table~\ref{table:num3}. In particular, the values of $m$ and $k$ are given in the first two columns, and the average CPU time (in seconds) and the average number of iterations over each set of 10 random instances for these methods are given in the rest of the columns. 
Besides, in Figure~\ref{fig:num3} we illustrate the behavior of the best relative Frank-Wolfe gap $\delta_t^*/\delta_0 := \min_{0\leq i \leq t}\delta_i/\delta_0$ and the relative objective function value $\varphi(U^t,V^t)/\varphi(U^0,V^0)$ with respect to CPU time on a single random instance of problem \eqref{eq:ex-entropy} with $m=n=300$, $\lambda=0.01$, $\alpha = 2$, and $k=5,10$, respectively.

One can observe that Algorithm~\ref{alg:main} significantly outperforms the conditional gradient method with line search proposed in \cite{Ghadimi19}. This is perhaps because: (i) the line search criterion of the conditional gradient method in \cite{Ghadimi19} explicitly depends on the targeted accuracy $\ep$, while the line search criterion of Algorithm~\ref{alg:main} does not; (ii) at each iteration, the initial trial step size in Algorithm~\ref{alg:main} is determined by using a constructive local quadratic upper approximation of the objective function, while the line search procedure in \cite{Ghadimi19} does not use such a novel scheme.

\begin{table}[p]
\centering
{
\begin{tabular}{cc||ll||ll}
\hline
  & & \multicolumn{2}{c||}{Average CPU time (sec)} & \multicolumn{2}{c}{Average number of iterations}\\
$m$ & $k$ & Algorithm~\ref{alg:main} & CGM-LS & Algorithm~\ref{alg:main} & CGM-LS \\ \hline\hline
100 & 5 & 0.71 & 6.17 & 172.4 & 1139.1 \\
    & 10 & 0.34 & 2.62 & 55.7 & 355.1 \\
200 & 5 & 2.26 & 17.09 & 329.8 & 2086.4 \\
    & 10 & 1.35 & 9.97 & 124.9 & 803.7 \\
300 & 5 & 6.41 & 49.24 & 624.1 & 4035.1 \\
    & 10 & 3.36 & 30.67 & 194.3 & 1398.2 \\
400 & 5 & 9.89 & 62.94 & 796.0 & 4208.8 \\
    & 10 & 4.87 & 41.01 & 230.3 & 1474.6 \\
500 & 5 & 14.74 & 179.24 & 904.3 & 7131.6 \\
    & 10 & 6.35 & 79.91 & 263.3 & 2588.0 \\
\hline\hline
\end{tabular}
}
\caption{Numerical results for problem~\eqref{eq:nmf}}
\label{table:num3}
\end{table}

\begin{figure}[p]
 \begin{minipage}{0.48\hsize}
 \begin{center}
  \includegraphics[width=70mm]{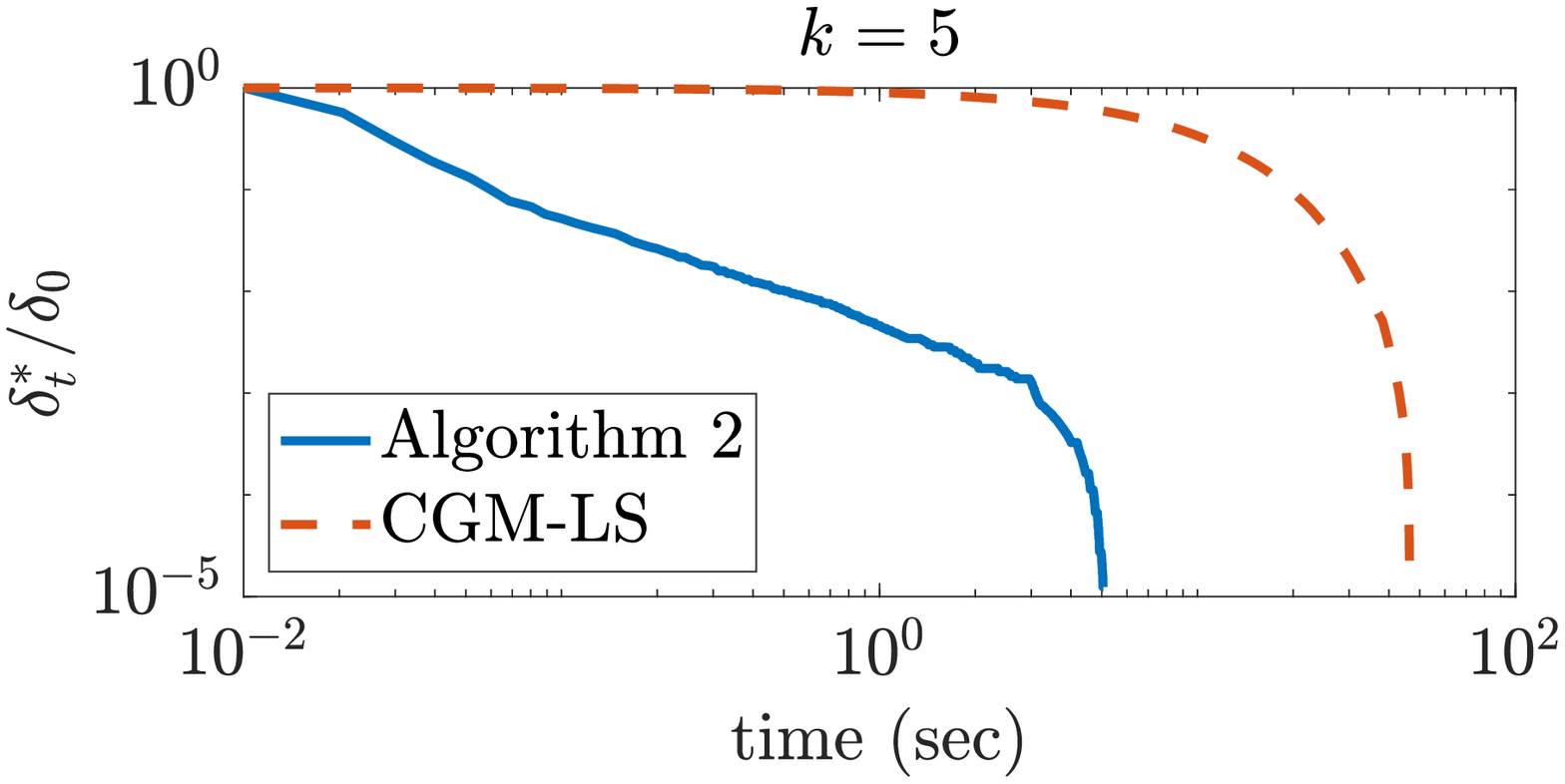}
 \end{center}
 \end{minipage}
 \begin{minipage}{0.48\hsize}
 \begin{center}
  \includegraphics[width=70mm]{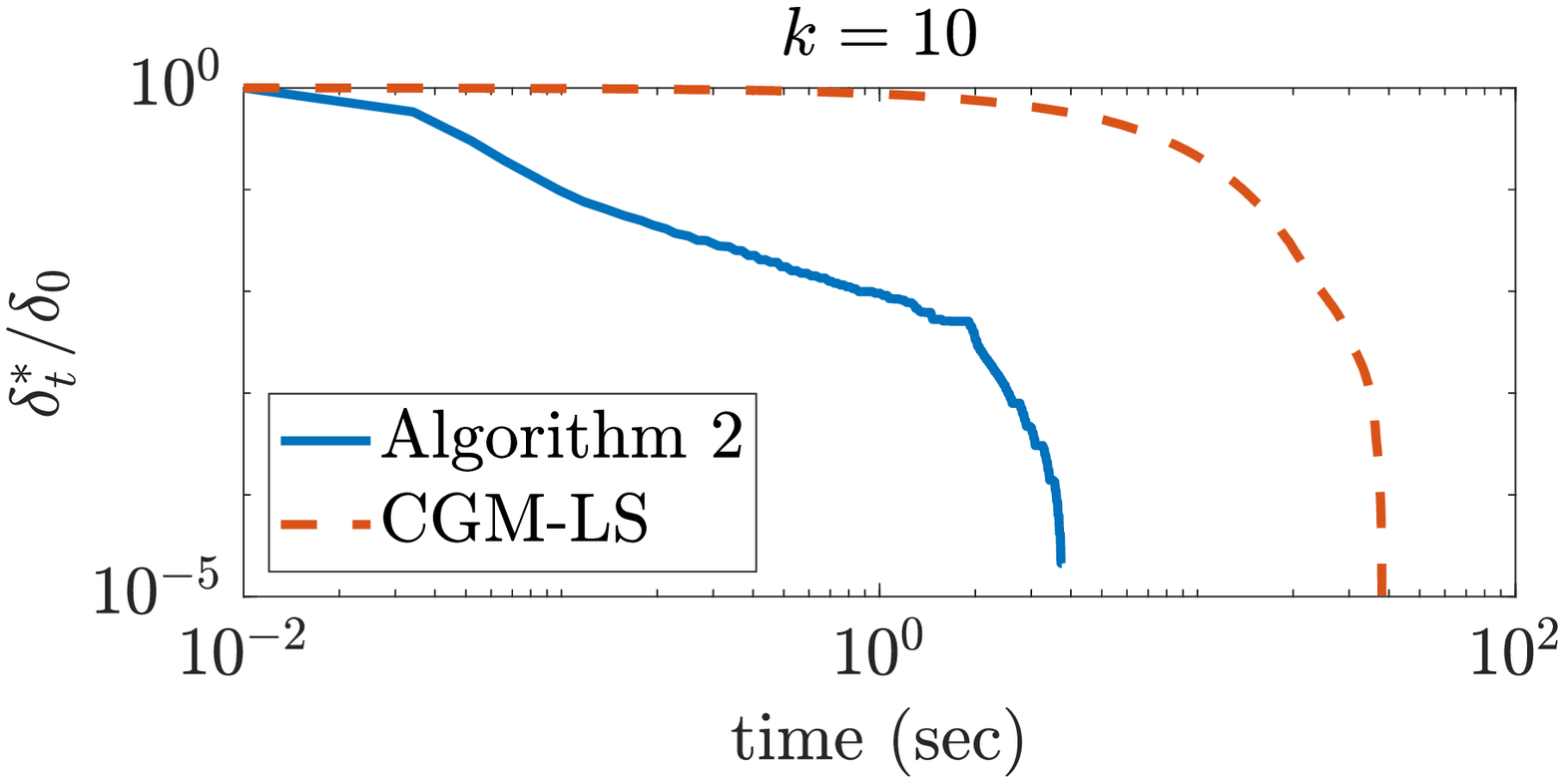}
 \end{center}
 \end{minipage}\\[2mm]
 \begin{minipage}{0.48\hsize}
 \begin{center}
  \includegraphics[width=70mm]{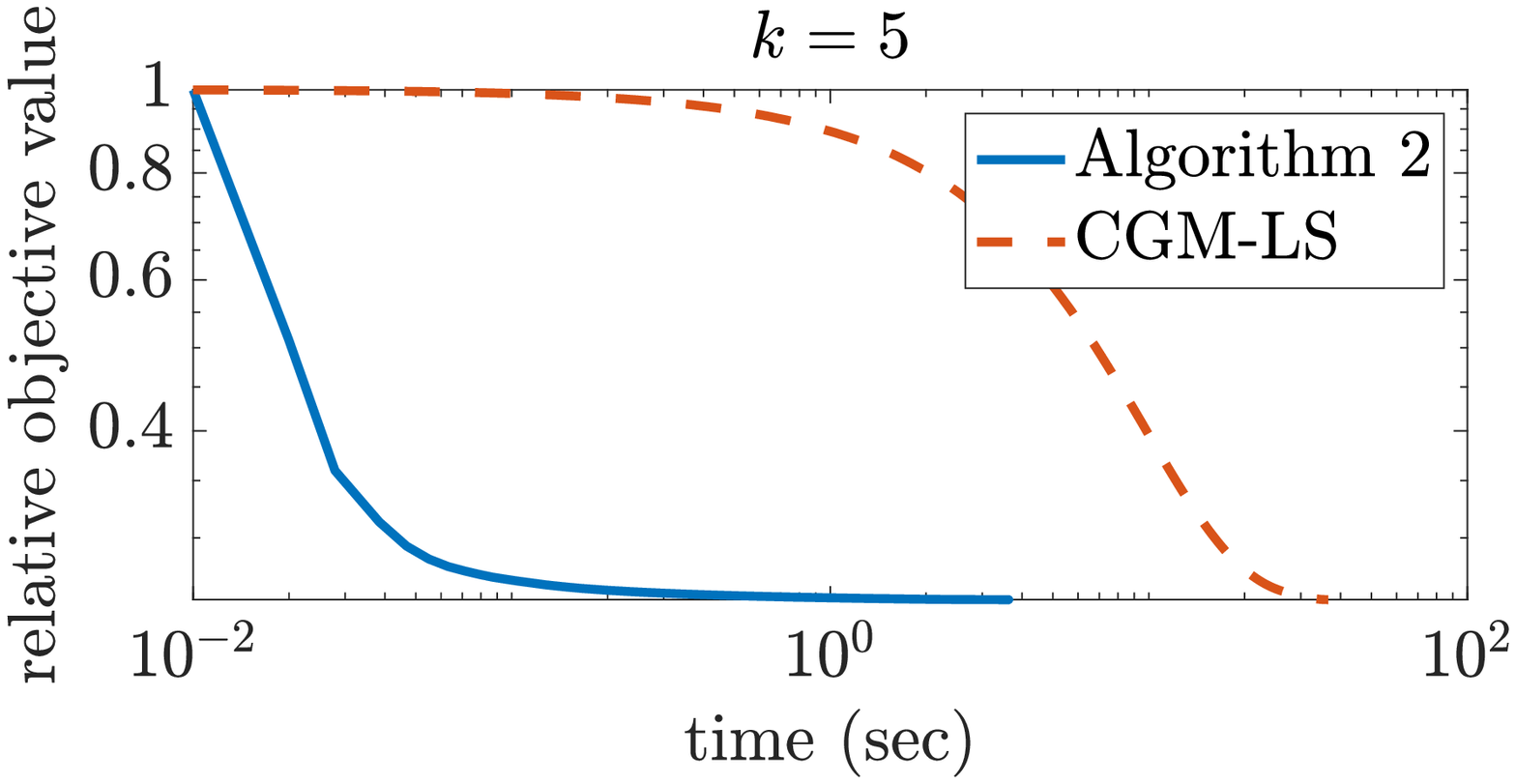}
 \end{center}
 \end{minipage}
 \begin{minipage}{0.48\hsize}
 \begin{center}
  \includegraphics[width=70mm]{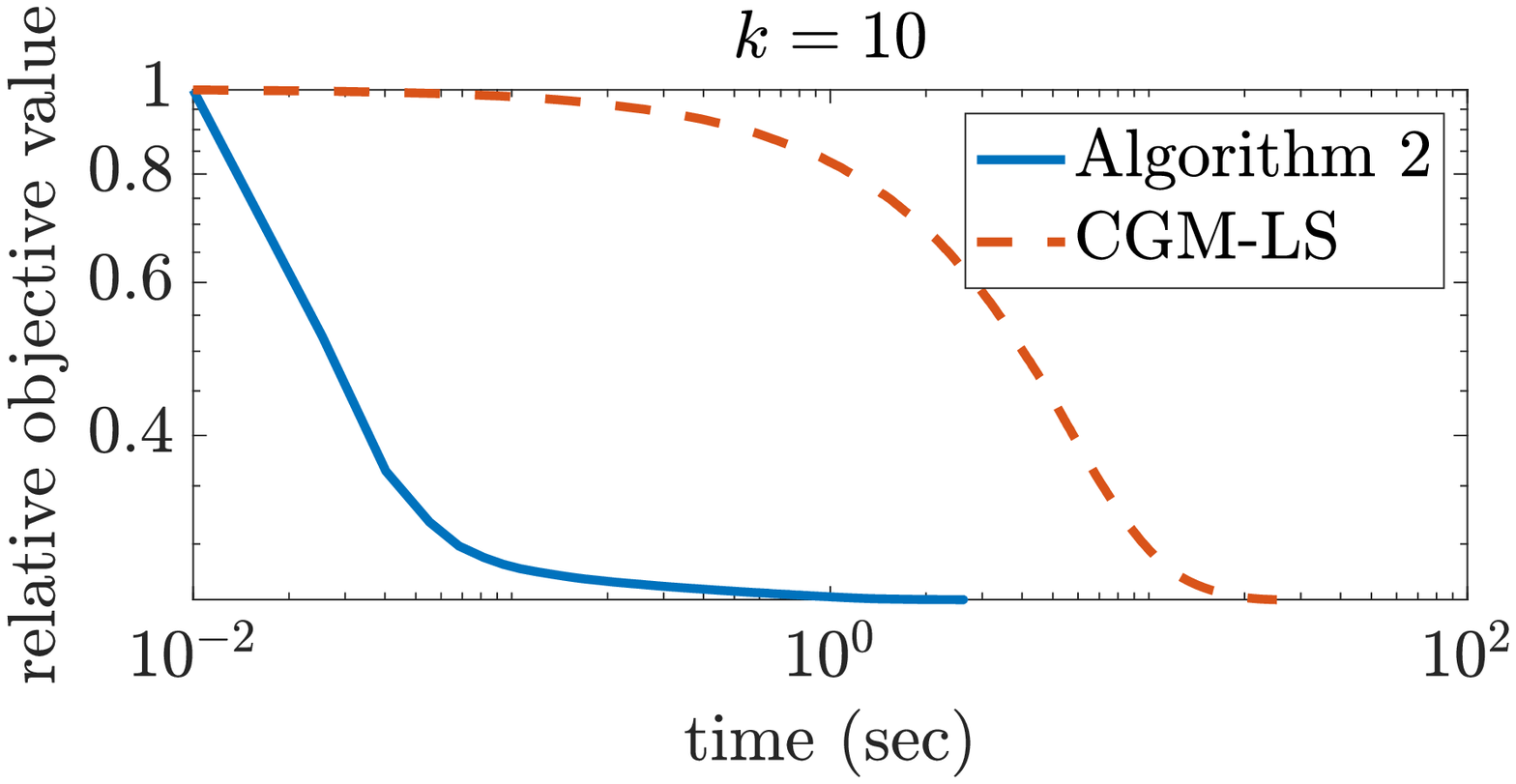}
 \end{center}
 \end{minipage}
 \caption{Numerical results on a single random instance  of problem \eqref{eq:nmf} with $m=n=300$, $\lambda=0.01$, $\alpha = 2$,  and $k=5,10$, respectively. These sub-figures illustrate the behavior of the best relative Frank-Wolfe gap $\delta_t^*/\delta_0 := \min_{0\leq i \leq t}\delta_i/\delta_0$ and the relative function value $\varphi(U^t,V^t)/\varphi(U^0,V^0)$ with respect to CPU time in seconds.}
 \label{fig:num3}
\end{figure}

\section{Proof of the Main Results}
\label{sec:proofs}

In this section, we provide a proof of our main results presented in Sections~\ref{sec:exact-ls} and \ref{sec:main}.

\subsection{Auxiliary Lemmas}
\label{sec:aux-lem}


In this subsection we establish some technical lemmas that will be used subsequently.

\begin{lemma}\label{lem:rec}
Suppose that $\{\beta_t\}$ and $\{\gamma_t\}$ are sequences of nonnegative real numbers such that 
\begin{equation}\label{eq:lem-rec}
\gamma_{t+1} \leq \gamma_t - c \beta_t \min\{1,\beta_t^\alpha/A\},\quad \forall t \geq 0
\end{equation}
for some constants $c \in (0,1)$, $\alpha\geq 0$ and $A>0$.
Then, $\underline{\gamma}=\lim_{t\to \infty}\gamma_t$ exists and the sequence $\beta_t^*=\min_{0\leq i\leq t}\beta_i$ satisfies
$$
\beta_t^*\leq \max\left\{ \frac{\gamma_0-\underline{\gamma}}{c(t+1)},~\left(\frac{A(\gamma_0-\underline{\gamma})}{c(t+1)}\right)^{\frac{1}{1+\alpha}} \right\},\quad \forall t \geq 0.
$$
In particular, we have $\beta_t^*\leq \ep$ whenever
$$
t \geq \frac{\gamma_0-\underline{\gamma}}{c \ep}\max\left\{1,\frac{A}{\ep^{\alpha}}\right\}.
$$
\end{lemma}

\begin{proof}
Since $\{\beta_t\} \subset \RR_+$, the relation  \eqref{eq:lem-rec} implies that $\{\gamma_t\}$ is non-increasing, which together with $\{\gamma_t\} \subset \RR_+$ further implies that $\underline{\gamma}=\lim_{t\to \infty}\gamma_t$ exists. In addition, by 
$\beta_t^*=\min_{0\leq i\leq t}\beta_i$ and \eqref{eq:lem-rec}, one can obtain that
\begin{equation}\label{gamma-delta-rec}
\gamma_{t+1} \leq \gamma_t - c\beta_t^*\min\{1,(\beta_t^*)^\alpha/A\},\quad \forall t \geq 0.
\end{equation}
Summing up these inequalities yields
$$\gamma_{t+1} \leq \gamma_0 - (t+1) c\beta_t^*\min\{1,(\beta_t^*)^\alpha/A\} ,\quad \forall t \geq 0.$$
By this and $\gamma_{t+1} \geq \underline{\gamma}$, we have $\beta_t^*\min\{1,(\beta_t^*)^\alpha/A\} \le(\gamma_0-\underline{\gamma})/(c(t+1))$, which implies the desired assertions.
\end{proof}


\begin{lemma}\label{lem:rec-conv}
Suppose that $\{\beta_t\}$ and $\{\gamma_t\}$ are sequences of nonnegative real numbers such that the recurrence \eqref{eq:lem-rec} holds for some constants $c \in (0,1)$, $\alpha\geq 0$ and $A>0$.
Assume additionally that $\beta_t \geq \gamma_t$ for all $t\geq 0$.
Let $\beta_t^*=\min_{0\leq i\leq t}\beta_i$.
Then the following statements hold.

\begin{enumerate}[(i)]
\item
If $\alpha=0$, then we have $\gamma_t \leq \overline\gamma_t$ for all $t \geq 0$ and $\beta_t^* \leq \overline\gamma_{\lfloor (t+2)/2\rfloor}$ for all $t \geq 2\max\{1,A\}/c$, where
$$
\overline{\gamma}_t = \gamma_0 \exp(-c \min\{1,A^{-1}\}\,t).
$$
Consequently, we have $\beta_t^* \leq \ep$ whenever
$$
t \geq \frac{2}{c}\max\{1,A\}\max\left\{1,\log\frac{\gamma_0}{\ep}\right\}.
$$
\item
If $\alpha>0$, then we have $\gamma_t \leq \overline\gamma_t$ for all $t \geq t_0$ and $\beta_t^* \leq (1+\alpha)^{\frac{1}{1+\alpha}} \overline\gamma_{\lfloor (t+t_0+1)/2 \rfloor} \leq e^{\frac{1}{e}}\overline\gamma_{\lfloor (t+t_0+1)/2 \rfloor}$ for all $t \geq t_0+2A/(c\gamma_{t_0}^\alpha)$,
where
\begin{equation}\label{lem:gamma-and-t0}
\overline{\gamma}_t = \left(\frac{1}{\gamma_{t_0}^{-\alpha}+A^{-1}c\alpha (t-t_0)}\right)^{\frac{1}{\alpha}},
\quad
t_0=\left\lceil \frac{1}{c}\left(\log \frac{\gamma_0}{cA^{1/\alpha}}\right)_+ \right\rceil.
\end{equation}
Consequently, we have $\beta_t^* \leq \ep$ whenever
\begin{equation}\label{eq:itr-compl-conv-sublin}
t \geq t_0+\frac{2A}{c\gamma_{t_0}^\alpha} \max\left\{1,~\frac{1}{\alpha}\left[\left(\frac{e^{\frac{1}{e}} \gamma_{t_0}}{\ep}\right)^\alpha - 1\right]\right\}.
\end{equation}
\end{enumerate}
\end{lemma}

\begin{proof}
(i) Consider the case $\alpha=0$.
By  \eqref{eq:lem-rec} and $\beta_t \geq \gamma_t \geq 0$ for all $t\geq 0$, one can obtain for all $t\geq 0$ that 
\begin{equation*}
\gamma_{t+1} \leq \gamma_t - c \gamma_t \min\{1,\gamma_t^\alpha/A\} = \gamma_t (1-c\min\{1,A^{-1}\}) \leq \gamma_t\exp(-c\min\{1,A^{-1}\}),
\end{equation*}
and hence $\gamma_t \leq \gamma_0\exp(-c\min\{1,A^{-1}\}t)=\overline\gamma_t$ for all $t \geq 0$.
By this, $\{\gamma_t\}\subset \RR_+$, and \eqref{gamma-delta-rec} with $\alpha=0$, one has that for any $k \leq t$,
\begin{align} 
c\min\{1,A^{-1}\}(t-k+1)\beta_t^* \leq \sum_{i=k}^t (c\min\{1,A^{-1}\}\beta_i^*) \leq \sum_{i=k}^t(\gamma_i-\gamma_{i+1}) = \gamma_k-\gamma_{t+1} \leq  \overline\gamma_k. \label{beta-bnd}
\end{align}
For convenience, let $T_0=2\max\{1,A\}/c$ and 
$T_1=\max\{1,A\}\log(\gamma_0/\ep)/c$. For any $t\geq T_0$, letting $k=\lfloor (t+2)/2\rfloor$ in \eqref{beta-bnd}, we obtain $\beta_t^* \leq \overline\gamma_{\lfloor (t+2)/2\rfloor}$ due to
$$c\min\{1,A^{-1}\}(t-k+1) \geq c\min\{1,A^{-1}\}t/2 = t/T_0 \geq 1.$$
Moreover, since $\overline{\gamma}_{\lfloor (t+2)/2\rfloor} \leq \ep$ holds if $\lfloor (t+2)/2\rfloor \geq T_1$, we have $\beta_t^*\leq \ep$ whenever
$$t\geq \max\{T_0,2T_1\} = \frac{2}{c}\max\{1,A\}\max\left\{1,\log\frac{\gamma_0}{\ep}\right\}.$$
Hence, statement (i) holds.

(ii)
We now consider the case $\alpha>0$.
It follows from the relation $\gamma_t\leq \beta_t$ and the monotonicity of the sequence
$\{\gamma_t\}$ that  $\gamma_t \le \beta^*_t$.
As long as $\beta_t^* > A^{1/\alpha}$, the relation \eqref{gamma-delta-rec} implies
\begin{equation}\label{gamma-rec1}
\gamma_{t+1} \leq \gamma_t - c\beta_t^* 
\le (1-c) \gamma_t \le \gamma_t \exp(-c).
\end{equation}
Claim that $\beta_{t_0}^* \leq A^{1/\alpha}$ for $t_0$ defined in \eqref{lem:gamma-and-t0}. Suppose for contradiction that $\beta_{t_0}^*>A^{1/\alpha}$.
Then, as \eqref{gamma-rec1} holds for $t=0,\ldots,t_0$, we have $\gamma_{t_0} \leq \gamma_0 \exp(-ct_0)$ and thus $\gamma_{t_0} \le cA^{1/\alpha}$ follows by the expression of $t_0$.
However, since $\{\gamma_t\}$ is nonnegative, the first inequality of \eqref{gamma-rec1} implies $\beta_{t_0}^* \leq c^{-1}\gamma_{t_0}\leq A^{1/\alpha}$, which leads to a contradiction.
Hence, $\beta_{t_0}^*\leq A^{1/\alpha}$ holds as claimed.

It follows from the monotonicity of $\{\beta_t^*\}$ that $\gamma_t \le \beta^*_t \le A^{1/\alpha}$ for all $t \ge t_0$. By this and \eqref{gamma-delta-rec}, one can obtain the recurrence 
\begin{equation}\label{gamma-rec2} 
\gamma_{t+1}  \leq \gamma_t - c(\beta_t^*)^{1+\alpha}/A \le \gamma_t - c\gamma_t^{1+\alpha}/A ,\quad \forall t \geq t_0.
\end{equation}
Then, by \cite[Lemma~4.1]{BLY14}, this recurrence implies the assertion
\begin{equation}\label{eq:gamma-ubd-pr}
\gamma_t
\leq (\gamma_{t_0}^{-\alpha}+A^{-1}c\alpha (t-t_0))^{-1/\alpha}
= \overline\gamma_t,  \quad \forall t \ge t_0.
\end{equation}

We next show that  $\beta_t^* \leq (1+\alpha)^{\frac{1}{1+\alpha}} \overline\gamma_{\lfloor (t+t_0+1)/2 \rfloor}$ for all $t \geq t_0+\frac{2A}{c\gamma_{t_0}^\alpha}$.
It follows from the first inequality in \eqref{gamma-rec2} and the monotonicity of 
$\{\beta^*_t\}$ that
\begin{equation}\label{gamma-delta-rec1}
c(\beta_t^*)^{1+\alpha}/A  \leq \gamma_i-\gamma_{i+1},\quad \forall t \ge t_0,~~i \leq t.
\end{equation}
Let $k=\lfloor (t+t_0+1)/2\rfloor$. Observe that $t-k+1 \geq k-t_0$. 
When $t\geq t_0+1$, we have $t \geq k \geq t_0+1$ and thus $\gamma_k \le \overline\gamma_k$ holds from \eqref{eq:gamma-ubd-pr}.
By these relations, $\gamma_{t+1} \ge 0$, and summing up the inequality \eqref{gamma-delta-rec1} for $i=k, \ldots, t$, one has
$$
(k-t_0)c(\beta_t^*)^{1+\alpha}/A \leq (t-k+1)c(\beta_t^*)^{1+\alpha}/A  \leq \gamma_k-\gamma_{t+1} \leq \gamma_k \leq \overline\gamma_k, \quad \forall t \ge t_0+1.
$$
In view of this and the expression of $\overline{\gamma}_t$, we obtain that for all $t \geq t_0+1$,
\begin{equation}\label{eq:b-pr-beta-ubd-gen}
\beta_t^* 
\leq \left(\frac{\overline\gamma_k}{A^{-1}c(k-t_0)}\right)^{\frac{1}{1+\alpha}}
= \left(\frac{\gamma_{t_0}^{-\alpha} + A^{-1}c\alpha(k-t_0)}{A^{-1}c(k-t_0)}\right)^{\frac{1}{1+\alpha}} \overline\gamma_k
=\theta_k\overline{\gamma}_k,
\end{equation}
where $\theta_k = \left(\frac{\gamma_{t_0}^{-\alpha} + A^{-1}c\alpha(k-t_0)}{A^{-1}c(k-t_0)}\right)^{\frac{1}{1+\alpha}}$.
Observe that $\theta_k$ is non-increasing and $\theta_k \downarrow 1$ as $k\to \infty$.
For convenience, let $T_2=t_0+2A/(c\gamma_{t_0}^\alpha)$.
Claim that
$\theta_k \leq (1+\alpha)^{\frac{1}{1+\alpha}}$ whenever $t \geq T_2$. 
Indeed, fix any $t \geq T_2$. 
By this and the expression of $k$, one can observe that 
$
k \geq t_0 + A/(c\gamma_{t_0}^\alpha),
$
which along with the expression of $\theta_k$ implies that $\theta_k \leq (1+\alpha)^{\frac{1}{1+\alpha}}$ holds as claimed.
Using this and \eqref{eq:b-pr-beta-ubd-gen}, we conclude that
$$\beta_t^* \leq (1+\alpha)^{\frac{1}{1+\alpha}}\overline{\gamma}_{\lfloor (t+t_0+1)/2\rfloor} \leq e^{\frac{1}{e}}\overline{\gamma}_{\lfloor (t+t_0+1)/2\rfloor},\quad \forall t \geq 
T_2,
$$
where the second inequality follows from $(1+\alpha)^{\frac{1}{1+\alpha}} \leq \max_{\theta>0}\theta^{-\theta}=e^{\frac{1}{e}}$.

Finally, by the expression of $\overline\gamma_t$, one can see that the relation $e^{\frac{1}{e}}\overline\gamma_{\lfloor (t+t_0+1)/2 \rfloor} \leq \ep$ holds if $t \geq T_3$, where
$$
T_3 = t_0 + \frac{2A}{c\gamma_{t_0}^\alpha \alpha}\left[\left(\frac{e^{\frac{1}{e}} \gamma_{t_0}}{\ep}\right)^\alpha-1\right].
$$
It then follows that $\beta_t^* \leq \ep$ holds whenever $t \geq \max\{T_2,T_3\}$ and hence \eqref{eq:itr-compl-conv-sublin} holds.
This completes the proof of statement (ii).
\end{proof}


The following lemma establishes the weak smoothness of the function $\frac{1}{p}\norm{Ax-b}_p^p$ for $p \in (1,2]$ which has been used in Section \ref{sec:num}.

\begin{lemma}\label{lem:lp-Holder}
For $p \in (1,2]$, $A \in \RR^{m\times n}$ and $b \in \RR^m$, the function $\phi(x)=\frac{1}{p}\norm{Ax-b}_p^p$ satisfies
$$
\norm{\nabla \phi(x) - \nabla \phi(y)}_2 \leq M_{p-1} \norm{x-y}_2^{p-1},
\quad \forall x,y \in \RR^n,
$$
where $M_{p-1}=2^{2-p} m^{\frac{(p-1)(2-p)}{2p}} \norm{A}_2^p$ and $\norm{A}_2=\max_{\norm{x}_2 \leq 1}\norm{Ax}_2$.
\end{lemma}
\begin{proof}
We first consider the univariate function $g(\tau)=\frac{1}{p}|\tau|^p$ for $\tau \in \RR$.
Its derivative is given by $g'(\tau)=|\tau|^{p-1}\sign(\tau)$. Claim that 
\begin{equation} \label{g-prop}
 |g'(\tau)-g'(\tau')| \leq 2^{2-p} |\tau-\tau'|^{p-1}, \quad \forall \tau, \tau' \in \RR.   
\end{equation}
Indeed, it suffices to show that
$$
|\alpha^{p-1}-\beta^{p-1}| \leq  |\alpha-\beta|^{p-1},\quad
\alpha^{p-1}+\beta^{p-1} \leq 2^{2-p} (\alpha+\beta)^{p-1}
$$
for every $\alpha,\beta \geq 0$. The first inequality follows from the fact $(x+y)^{p-1} \leq x^{p-1} + y^{p-1}$ for $x,y \geq 0$ and the second one holds due to the concavity property $[(\alpha+\beta)/2]^{p-1} \geq (\alpha^{p-1}+\beta^{p-1})/2$. Hence, \eqref{g-prop} holds as claimed.

Let $h(z)=\frac{1}{p}\norm{z}_p^p$ for any $z \in \RR^m$. Notice that $h(z)=\sum_{i=1}^mg(z_i)$, which together with \eqref{g-prop} implies that $\norm{\nabla h(x)-\nabla h(y)}_p \leq 2^{2-p}\norm{x-y}_p^{p-1}$. 
Also, observe that $\phi(x)=h(Ax-b)$. Using these and $\norm{z}_2 \leq \norm{z}_p \leq m^{1/p-1/2}\norm{z}_2$ for any $z\in \RR^m$, we obtain that
\begin{align*}
\norm{\nabla \phi(x)-\nabla \phi(y)}_2 &\leq \norm{A^T}_2 \norm{\nabla h(Ax-b)-\nabla h(Ay-b)}_2 \leq \norm{A}_2 \norm{\nabla h(Ax-b)- \nabla h(Ay-b)}_p\\
& \leq 2^{2-p}\norm{A}_2 \norm{A(x-y)}_p^{p-1} \leq 2^{2-p} m^{(p-1)(1/p-1/2)} \norm{A}_2 \norm{A(x-y)}_2^{p-1}\\
& \leq 2^{2-p} m^{(p-1)(2-p)/(2p)}\norm{A}_2^p \norm{x-y}_2^{p-1}. 
\end{align*}
Hence, the conclusion holds as desired.
\end{proof}

\subsection{Proof of the Main Results in Section~\ref{sec:exact-ls}}
\label{sec:proof-1}

In this subsection, we prove Theorems~\ref{rate-exact-ls-1} and \ref{rate-exact-ls-2}.
Before proceeding, we establish a descent property for the sequence $\{\varphi(x_t)\}$.

\begin{lemma}\label{lem:alg1-rec}
Let the sequences $\{x_t\}$, $\{\delta_t\}$ and $\{v_t\}$ be generated in Algorithm~\ref{alg:FW-exact-ls}. Suppose that Assumption~\ref{assump:fg} holds and that $\delta_t> 0$ for all $t \ge 0$.\footnote{If $\delta_t=0$ for some $t \ge 0$, $x_t$ is already a stationary point of problem \eqref{main-prob} and Algorithm~\ref{alg:FW-exact-ls} shall be terminated.}
Then we have 
\begin{equation}\label{exact-ls-bound-1}
\varphi(x_{t+1}) \leq \varphi(x_t) - \frac{\nu}{1+\nu} \delta_t\min\left\{1, \left(\frac{\delta_t}{M_\nu \norm{x_t-v_t}^{1+\nu}}\right)^{\frac{1}{\nu}}\right\}, \quad \forall t \ge 0. 
\end{equation}
\end{lemma}

\begin{proof}
By the H\"older continuity of $\nabla{f}$ (see \eqref{assump:Holder}), we have
\begin{equation}\label{eq:Holder}
f(y) \leq f(x)+\innprod{\nabla{f}(x)}{y-x} + \frac{M_\nu}{1+\nu}\norm{x-y}^{1+\nu},\quad \forall x,y \in \dom g.
\end{equation}
 By the convexity of $g$, and \eqref{eq:Holder} with $y=(1-\tau)x_t+\tau v_t$ and $x=x_t$, one can obtain that for any $\tau \in [0,1]$,
\begin{align}
& \varphi((1-\tau)x_t + \tau v_t) \nonumber\\
& \leq 
f(x_t) + \innprod{\nabla{f}(x_t)}{(1-\tau)x_t+\tau v_t -x_t} + \frac{M_\nu}{1+\nu}\norm{(1-\tau)x_t+\tau v_t -x_t}^{1+\nu}\nonumber \\
& \quad  + g((1-\tau)x_t+\tau v_t)
 \nonumber\\
& \leq f(x_t) -\tau\innprod{\nabla{f}(x_t)}{x_t-v_t} + \tau^{1+\nu}\frac{M_\nu}{1+\nu}\norm{x_t-v_t}^{1+\nu}
+(1-\tau)g(x_t) + \tau g(v_t) \nonumber\\
& = \varphi(x_t) -\tau \delta_t + \tau^{1+\nu} \frac{M_\nu}{1+\nu}\norm{x_t-v_t}^{1+\nu}. \label{eq:step size-Holder-bound}
\end{align}
Letting $\tau=\tau_t$ in \eqref{eq:step size-Holder-bound}, and using the expression of $\tau_t$ and $x_{t+1}$, we obtain that for any $t \ge 0$, 
\begin{align*}
\varphi(x_{t+1})  &\leq \varphi(x_t) - \tau_t\delta_t + \tau^{1+\nu}_t \frac{M_\nu}{1+\nu}\norm{x_t-v_t}^{1+\nu} \\
& \leq \varphi(x_t) - \frac{\nu}{1+\nu} \delta_t\min\left\{1, \left(\frac{\delta_t}{M_\nu \norm{x_t-v_t}^{1+\nu}}\right)^{\frac{1}{\nu}}\right\}.
\end{align*}
\end{proof}

We are now ready to prove Theorems~\ref{rate-exact-ls-1} and \ref{rate-exact-ls-2}.

\vskip 0.1in
\noindent
\textbf{Proof of Theorem~\ref{rate-exact-ls-1}}~~
 Let  the sequences $\{x_t\}$ and $\{v_t\}$ be generated in Algorithm~\ref{alg:FW-exact-ls}. One can observe that $x_t, v_t \in \dom{g}$ for all $t\geq 0$. It then follows that $\|x_t-v_t\| \le D_g$. By this and 
Lemma~\ref{lem:alg1-rec},  one can obtain 
\begin{equation}\label{exact-ls-bound-2}
\varphi(x_{t+1})  \leq \varphi(x_t)  - \frac{\nu}{1+\nu}\delta_t \min\left\{1, \left(\frac{\delta_t}{M_\nu D_g^{1+\nu}}\right)^{\frac{1}{\nu}}\right\}, \quad \forall t \ge 0.
\end{equation}

(i) It follows from \eqref{exact-ls-bound-2} that $\{\varphi(x_t)\}$ is non-increasing, which, together with the fact that $\varphi(x_t) \ge \varphi^*$ for all $t \ge 0$, implies that $\varphi_* =\lim_{t\to \infty}\varphi(x_t)$ exists. In addition,  one can observe from  
\eqref{exact-ls-bound-2} that the recurrence \eqref{eq:lem-rec} holds for $\beta_t=\delta_t$, $\gamma_t=\varphi(x_t)-\varphi_*$, 
$\alpha=1/\nu$, $c=\nu/(1+\nu)$, and $A=M_\nu^{\frac{1}{\nu}}D_g^{\frac{1+\nu}{\nu}}$. The inequality \eqref{delta*-uppbnd-1} then directly follows from Lemma~\ref{lem:rec}.

(ii) One can observe from \eqref{exact-ls-bound-2} that the recurrence \eqref{eq:lem-rec} holds for $\beta_t=\delta_t$, $\gamma_t=\varphi(x_t)-\varphi^*$,  
$\alpha=1/\nu$, $c=\nu/(1+\nu)$, and $A=M_\nu^{\frac{1}{\nu}}D_g^{\frac{1+\nu}{\nu}}$.
In addition, $\beta_t \ge \gamma_t$ due to Lemma~\ref{lem:FW-gap-bound-convex}. The conclusion of this statement then immediately follows from Lemma~\ref{lem:rec-conv}\,(ii).
\hfill$\blacksquare$

\vskip 0.1in
\noindent
\textbf{Proof of Theorem~\ref{rate-exact-ls-2}}~~
Let  the sequences $\{x_t\}$ and $\{v_t\}$ be generated in Algorithm~\ref{alg:FW-exact-ls}. One can observe that $x_t, v_t \in \dom{g}$ for all $t\geq 0$. By this, the expression of $\delta_t$, and Assumption~\ref{assump:g}, one has
\begin{equation} \label{delta_t-lwrbnd}
\delta_t = \innprod{\nabla{f}(x_t)}{x_t}+g(x_t) - \innprod{\nabla{f}(x_t)}{v_t}-g(v_t) \geq \frac{\kappa}{\rho} \norm{x_t-v_t}^\rho, \quad \forall t \ge 0.
\end{equation}
Using this inequality, we can obtain that
$$
\frac{\delta_t}{M_\nu \norm{x_t-v_t}^{1+\nu}} = 
\frac{1}{M_\nu}\left(\frac{\delta_t}{\norm{x_t-v_t}^\rho}\right)^{\frac{1+\nu}{\rho}}
\delta_t^{1-\frac{1+\nu}{\rho}}
\geq
\frac{1}{M_\nu}\left(\frac{\kappa}{\rho}\right)^{\frac{1+\nu}{\rho}}\delta_t^{1-\frac{1+\nu}{\rho}},
$$
which together with \eqref{exact-ls-bound-1} yields
\begin{equation}\label{inexact-ls-bound}
\varphi(x_{t+1}) \leq \varphi(x_t)  - \frac{\nu}{1+\nu}\delta_t \min\left\{1, \left(\frac{\kappa^{\frac{1+\nu}{\rho}}\delta_t^{1-\frac{1+\nu}{\rho}}}{\rho^{\frac{1+\nu}{\rho}}M_\nu}\right)^{\frac{1}{\nu}}\right\}, \quad \forall t \ge 0.
\end{equation}

(i) By \eqref{inexact-ls-bound} and a similar argument as in the proof of Theorem~\ref{rate-exact-ls-1} (i),  one can see that $\{\varphi(x_t)\}$ is non-increasing and $\varphi_* =\lim_{t\to \infty}\varphi(x_t)$ exists. In addition,  one can observe from  
\eqref{inexact-ls-bound} that the recurrence \eqref{eq:lem-rec} holds for $\beta_t=\delta_t$, $\gamma_t=\varphi(x_t)-\varphi_*$, 
$\alpha=(\rho-1-\nu)/(\rho\nu)$, $c=\nu/(1+\nu)$, and $A=\left(\frac{\rho}{\kappa}\right)^{\frac{1+\nu}{\rho\nu}}M_\nu^{\frac{1}{\nu}}$. The inequality \eqref{delta*-uppbnd-2} then directly follows from Lemma~\ref{lem:rec}.

(ii) Let $\alpha=(\rho-1-\nu)/(\rho\nu)$ and $c=\nu/(1+\nu)$. One can observe from  
\eqref{inexact-ls-bound} that the recurrence \eqref{eq:lem-rec} holds for such $\alpha$, $c$, $\beta_t=\delta_t$, $\gamma_t=\varphi(x_t)-\varphi^*$, and $A=\left(\frac{\rho}{\kappa}\right)^{\frac{1+\nu}{\rho\nu}}M_\nu^{\frac{1}{\nu}}$. Also, $\beta_t \ge \gamma_t$ due to Lemma~\ref{lem:FW-gap-bound-convex}. In addition, by $\nu\in (0,1]$, $\rho \ge 2$, and the expression of $\alpha$, it is not hard to see that $\alpha=0$ if and only if $\nu=1$ and $\rho=2$. Also, one can see from the expression of $c$ and $A$ that $c=1/2$ and $A=2M_1/\kappa$ when $\nu=1$ and $\rho=2$. The conclusion of this statement then follows from these observations and  Lemma~\ref{lem:rec-conv}.
\hfill$\blacksquare$

\subsection{Proof of the Main Results in Section~\ref{sec:main}}
\label{sec:proof-2}

In this subsection, we prove Theorems~\ref{lem:L-est}, \ref{th:main1}, and \ref{th:main2}.


\vskip 0.1in
\noindent
\textbf{Proof of Theorem~\ref{lem:L-est}}~~
(i)
For any $\tau \in [0,1]$, by the convexity of $g$, the expression of $\delta_t$, and  \eqref{eq:Holder-approx} with $x=x_t$, 
$y=(1-\tau)x_t + \tau v_t$ and $\ep=\tau\delta_t/2$, one has 
\begin{align*}
& \varphi((1-\tau)x_t + \tau v_t)=f((1-\tau)x_t + \tau v_t)+g((1-\tau)x_t + \tau v_t)\\
& \leq 
f(x_t) + \innprod{\nabla{f}(x_t)}{(1-\tau)x_t+\tau v_t -x_t} + \frac{L(\tau\delta_t/2)}{2}\norm{(1-\tau)x_t+\tau v_t -x_t}^2 + \frac{\tau\delta_t}{2}\\
& \quad + g((1-\tau)x_t+\tau v_t)
\\
& \leq f(x_t) -\tau\innprod{\nabla{f}(x_t)}{x_t-v_t} + \tau^2\frac{L(\tau\delta_t/2)}{2}\norm{x_t-v_t}^2
+(1-\tau)g(x_t) + \tau g(v_t) + \frac{\tau\delta_t}{2} \\
& = \varphi(x_t) -\frac{1}{2} \tau\delta_t + \tau^2 \frac{L(\tau\delta_t/2)}{2}\norm{x_t-v_t}^{2}.
\end{align*}
Letting $\tau=\tau_t^{(i)}$ in this inequality yields
$$
\varphi(x_{t+1}^{(i)}) \leq \varphi(x_t) -\frac{1}{2} \tau_t^{(i)}\delta_t + (\tau_t^{(i)})^2 \frac{L(\tau_t^{(i)}\delta_t/2)}{2}\norm{x_t-v_t}^{2}.
$$
Hence, \eqref{phi-reduct} holds if
\begin{equation}\label{eq:ls-suff}
L_t^{(i)} \geq L(\tau_t^{(i)}\delta_t/2)=\max\left\{L(\delta_t/2), L\left(\frac{\delta_t^2}{4L_t^{(i)}\norm{x_t-v_t}^2}\right)\right\},
\end{equation}
where the equality follows from the expression of $\tau_t^{(i)}$ and the fact that $L(\cdot)$ is non-increasing. By \eqref{eq:L-ep}, one can verify that 
\[
L_t^{(i)} \geq L\left(\frac{\delta_t^2}{4L_t^{(i)}\norm{x_t-v_t}^2}\right) ~~\Longleftrightarrow ~~
L_t^{(i)} \geq L\left(\frac{\delta_t^2}{4\norm{x_t-v_t}^2}\right)^{\frac{1+\nu}{2\nu}},
\]
which together with \eqref{eq:ls-suff} implies that 
\[
L_t^{(i)} \geq L(\tau_t^{(i)}\delta_t/2) ~~\Longleftrightarrow ~~ L_t^{(i)} \geq \max\left\{L(\delta_t/2), L\left(\frac{\delta_t^2}{4\norm{x_t-v_t}^2}\right)^{\frac{1+\nu}{2\nu}}\right\} = \widetilde{L}_t.
\]
Hence, \eqref{phi-reduct} holds if $L_t^{(i)}\geq \widetilde{L}_t$.

(ii) In the $t$-th outer iteration, let $i_t \geq 0$ denote the final iteration counter for the adaptive line search loop, and let $\widetilde{L}_t^*=\max_{0\leq i\leq t}\widetilde{L}_i$.
For $t\geq 0$ and $s\in\{0,\ldots,t\}$, we first show that
\begin{equation}\label{eq:L-bound-renew}
L_{s-1} \leq 2\widetilde{L}_t^* ~~\Longrightarrow~~L_s \leq 2\widetilde{L}_t^*.
\end{equation}
Indeed, if $i_s=0$, one can observe that $L_s = L_{s-1}/2$, which immediately implies that \eqref{eq:L-bound-renew}  holds. Now we suppose $i_s>0$. 
It then follows that the adaptive line search loop fails to terminate at the inner iteration $i_s-1$, which along with statement (i) implies that $L_s/2=L_s^{(i_s-1)}< \widetilde{L}_s$. 
It then follows  that $L_s< 2\widetilde{L}_s \le 2\widetilde{L}_t^*$. Hence, \eqref{eq:L-bound-renew} holds as desired. By these arguments, one can also observe that $L_t = L_{t-1}/2$ whenever $t \not\in \cT$, where
$$\cT = \{t \in \ZZ_+: L_{t-1} \leq 2 \widetilde{L}_t^*\}.$$
Due to this observation and the fact that $\widetilde{L}_t^*\ge \widetilde{L}_0>0$ for all $t \ge 0$, the set $\cT$ must be nonempty. Then, $\tilde t_0=\min\{t: t\in \cT\}$ is well-defined.
Since $t \in \cT$ implies $t+1 \in \cT$ due to \eqref{eq:L-bound-renew}, we see that $\cT=\{\tilde t_0,\tilde t_0+1,\ldots\}$. In addition, \eqref{eq:L-bound-renew} implies $\cT \subset \{t\in \ZZ_+: L_t \leq 2 \widetilde{L}_t^*\}$. Hence, we obtain that
$$
L_t \leq 2\widetilde{L}_t^*,\quad \forall t \geq \tilde{t}_0.
$$
To complete the proof, it suffices to show $\tilde t_0 \leq (\log_2( L_{-1}/\widetilde{L}_0))_+$. Indeed, it holds trivially if $\tilde t_0=0$. Now we suppose $\tilde t_0>0$. Recall that $L_t = L_{t-1}/2$ whenever $t \not\in \cT$. It then follows that $L_t = L_{t-1}/2$ for $t = 0,\ldots,\tilde t_0-1$. 
Hence, we have 
\[
L_{-1}/2^{\tilde t_0-1} = L_{\tilde t_0-2} > 2\widetilde{L}_{\tilde t_0-1}^* \geq 2\widetilde{L}_0,
\] 
which yields $\tilde t_0 \leq \log_2 (L_{-1}/\widetilde{L}_0)$. Thus, $\tilde t_0 \leq (\log_2( L_{-1}/\widetilde{L}_0))_+$ holds as desired. 

(iii) We first show $\widetilde{L}_t \leq \overline{L}(\delta_t)$. Indeed, by \eqref{eq:L-ep} and the expression of $\widetilde{L}_t$, one has
\begin{equation}\label{eq:L_t}
\widetilde{L}_t  = \max\left\{\left(\frac{1-\nu}{1+\nu}\frac{1}{\delta_t}\right)^{\frac{1-\nu}{1+\nu}}M_\nu^{\frac{2}{1+\nu}}, \left(\frac{2(1-\nu)}{1+\nu}\right)^{\frac{1-\nu}{2\nu}}\left(\frac{\norm{x_t-v_t}}{\delta_t}\right)^{\frac{1-\nu}{\nu}}M_\nu^{\frac{1}{\nu}}\right\}.
\end{equation}
 Since $x_t, v_t \in \dom{g}$, one can observe that $\norm{x_t-v_t}/\delta_t \leq D_g/\delta_t$ if $\dom{g}$ is bounded. Also, if Assumption~\ref{assump:g} holds, it follows from \eqref{delta_t-lwrbnd} that 
\[
\frac{\norm{x_t-v_t}}{\delta_t} = \left(\frac{\norm{x_t-v_t}^\rho}{\delta_t}\right)^{\frac{1}{\rho}}\delta_t^{\frac{1}{\rho}-1} \leq \left(\frac{\rho}{\kappa}\right)^{\frac{1}{\rho}}\delta_t^{\frac{1}{\rho}-1}.
\]
Then, $\widetilde{L}_t \leq \overline{L}(\delta_t)$ holds due to \eqref{eq:L_t} and the last two 
inequalities. By this, $\min_{0\leq i\leq t}\delta_i\geq \ep$, and the fact that $\overline{L}(\cdot)$ is non-increasing, we obtain that
$$
\widetilde{L}_t^* = \max_{0\leq i\leq t}\widetilde{L}_i \leq \max_{0\leq i\leq t}\overline{L}(\delta_i) \leq \overline{L}(\ep) .
$$
In addition, from the proof of statement (ii), we can observe that $L_t \le \max\{L_{-1}/2,2\widetilde{L}_t^*\}$ for all $t \ge 0$.
Also, by the definition of $i_s$, one can see that $L_s = 2^{i_s - 1}L_{s-1}$ for $s \ge 0$, and 
the total number of inner loops performed by the adaptive line search procedure until the $t$-th  iteration of Algorithm~\ref{alg:main}
is given by $\sum_{s=0}^t(1+i_s)$. By these observations, one can have
\begin{align*}
\sum_{s=0}^t(1+i_s) &= \sum_{s=0}^t \left(2+\log_2 \frac{L_s}{L_{s-1}}\right) = 
2(t+1)+\log_2 \frac{L_t}{L_{-1}} \\
& \leq 2(t+1)+\log_2 \frac{\max\{L_{-1}/2,2\widetilde{L}_t^*\}}{L_{-1}}
\leq 2t+2+[\log_2(2\overline{L}(\ep)/L_{-1})]_+,
\end{align*}
and hence the conclusion holds.
\hfill$\blacksquare$

Before proving Theorems~\ref{th:main1} and \ref{th:main2}, we establish a lemma that will be used shortly.

\begin{lemma} \label{lem:alg2-rec}
Let the sequences $\{x_t\}$, $\{\delta_t\}$ and $\{v_t\}$ be generated in Algorithm~\ref{alg:main}. Suppose that Assumption~\ref{assump:fg} holds and that $\delta_t>0$ for all $t \ge 0$. Let $\tilde t_0=\lceil(\log_2 (L_{-1}/\widetilde{L}_0))_+\rceil$, $\delta_t^* = \min_{0\leq i\leq t}\delta_t$,  and $\widetilde{L}_t$ and $\overline{L}(\cdot)$ be defined in \eqref{eq:def-tilde-Lt} and \eqref{eq:def-tilde-L}, respectively. 
Then it holds that 
\begin{equation}\label{eq:alg2-rec}
\varphi(x_{t+1}) \leq  \varphi(x_t) - \frac{\delta_t^*}{4} \min\left\{1, C_t\right\}, \quad \forall t \ge \tilde t_0,
\end{equation}
where 
\begin{equation}\label{C_t}
C_t := \frac{1}{2\overline{L}(\delta_t^*)}\min_{0\leq i \leq t}\frac{\delta_i}{\norm{x_i-v_i}^2}, \quad \forall t \ge 0.
\end{equation}
\end{lemma}

\begin{proof}
One can observe from Algorithm~\ref{alg:main} that 
\[
\tau_t=\min\left\{1, \frac{\delta_t}{2 L_t\norm{x_t-v_t}^2}\right\}, \quad 
\varphi(x_{t+1}) \leq \varphi(x_t)-\frac{1}{2}\tau_t\delta_t + \frac{1}{2}L_t\tau_t^2\norm{x_t-v_t}^2, \quad \forall t \ge 0.
\]
It then follows that
\begin{equation}\label{eq:alg2-rec1}
\varphi(x_{t+1}) \leq \varphi(x_t) - \frac{\delta_t}{4} \min\left\{1, \frac{\delta_t}{L_t\norm{x_t-v_t}^2}\right\},\quad \forall t \geq 0.
\end{equation}
Recall from the proof of Theorem~\ref{lem:L-est}\,(iii) that $\widetilde{L}_t \le \overline{L}(\delta_t)$ for all $t \ge 0$. Using this, Theorem~\ref{lem:L-est}\,(ii), $\delta_t^* = \min_{0\leq i\leq t}\delta_t$, and the monotonicity of $\overline{L}(\cdot)$, we have
$$L_t\leq 2\max_{0\leq i\leq t}\widetilde{L}_i \le 2\max_{0\leq i\leq t}\overline{L}(\delta_i)
= 2\overline{L}(\delta_t^*),\quad \forall t \geq \tilde t_0.
$$
The conclusion then follows from this, \eqref{eq:alg2-rec1}, and $\delta_t^* = \min_{0\leq i\leq t}\delta_t$.
\end{proof}

We are now ready to prove Theorems~\ref{th:main1} and \ref{th:main2}.

%

\vskip 0.1in
\noindent
\textbf{Proof of Theorem~\ref{th:main1}}~~
One can observe that $x_t, v_t \in \dom{g}$. It then follows that $\|x_t-v_t\| \le D_g$ for all $t\geq 0$. By this and $\delta_t^* = \min_{0\leq i\leq t}\delta_i$, one has
\begin{equation}\label{eq:pr-main-th-ratio-bound1}
\frac{\delta_i}{\norm{x_i-v_i}^2} \geq \frac{\delta_t^*}{D_g^2}, \quad \forall 0\le i \le t.
\end{equation}
Let $C_t$ be defined in \eqref{C_t}. We next bound $C_t$ from below by considering two cases in view of the definition of $\overline{L}(\cdot)$ in \eqref{eq:def-tilde-L}.

Case 1) $\overline{L}(\delta_t^*)=\left(\frac{2(1-\nu)}{1+\nu}\right)^{\frac{1-\nu}{2\nu}}M_\nu^{\frac{1}{\nu}}
\left(\frac{D_g}{\delta_t^*}\right)^{\frac{1-\nu}{\nu}}$. By this, \eqref{C_t} and \eqref{eq:pr-main-th-ratio-bound1}, we obtain that 
\[
C_t \geq
\frac{1}{2}
\left(\frac{2(1-\nu)}{1+\nu}\right)^{-\frac{1-\nu}{2\nu}}
M_\nu^{-\frac{1}{\nu}}\left(\frac{D_g}{\delta_t^*}\right)^{-\frac{1-\nu}{\nu}}
\cdot
\frac{\delta_t^*}{D_g^2}
=
2^{-\frac{1+\nu}{2\nu}}\left(\frac{1-\nu}{1+\nu}\right)^{-\frac{1-\nu}{2\nu}}
\left(\frac{\delta_t^*}{M_\nu D_g^{1+\nu}}\right)^{\frac{1}{\nu}}
=:D_t.
\]

Case 2) $\overline{L}(\delta_t^*) = \left(\frac{1-\nu}{1+\nu}\frac{1}{\delta_t^*}\right)^{\frac{1-\nu}{1+\nu}}{M_\nu^{\frac{2}{1+\nu}}}$. By this, \eqref{C_t} and \eqref{eq:pr-main-th-ratio-bound1}, one has 
\[
C_t \geq
\frac{1}{2}
\left(\frac{1-\nu}{1+\nu}\frac{1}{\delta_t^*}\right)^{-\frac{1-\nu}{1+\nu}}{M_\nu^{-\frac{2}{1+\nu}}}
\cdot
\frac{\delta_t^*}{D_g^2}
=
\frac{1}{2}\left(\frac{1-\nu}{1+\nu}\right)^{-\frac{1-\nu}{1+\nu}}
\left(\frac{\delta_t^*}{M_\nu D_g^{1+\nu}}\right)^{\frac{2}{1+\nu}}=D_t^{\frac{2\nu}{1+\nu}}.
\]
Combining these two cases, and using \eqref{eq:def-tilde-L} and \eqref{C_t}, we conclude that $C_t \geq \min\{D_t, D_t^{\frac{2\nu}{1+\nu}}\}$.
By this and $2\nu/(1+\nu) \leq 1$, one can observe that
$$
\min\{1,C_t\} \geq \min\{1, D_t\}, \quad \forall t \ge 0.
$$
In view of this, \eqref{eq:alg2-rec}, and the expression of $D_t$, one has 
\begin{align}\label{exact-ls-bound-3}
\varphi(x_{t+1})
&\leq \varphi(x_t) - \frac{\delta_t^*}{4}\min\left\{1, 2^{-\frac{1+\nu}{2\nu}}\left(\frac{1-\nu}{1+\nu}\right)^{-\frac{1-\nu}{2\nu}}
\left(\frac{\delta_t^*}{M_\nu D_g^{1+\nu}}\right)^{\frac{1}{\nu}}\right\}, \nonumber\\
&\leq \varphi(x_t) - \frac{\delta_t^*}{4}\min\left\{1, 
\left(\frac{\delta_t^*}{2M_\nu D_g^{1+\nu}}\right)^{\frac{1}{\nu}}\right\},
\quad \forall t \geq \tilde t_0,
\end{align}
where the last inequality is due to $2^{\frac{1+\nu}{2\nu}} \leq 2^{\frac{1+1}{2\nu}}$ and $\frac{1-\nu}{1+\nu}\leq 1$.

(i) It follows from \eqref{eq:alg2-rec1} that $\{\varphi(x_t)\}$ is non-increasing, which, together with the fact that $\varphi(x_t) \ge \varphi^*$ for all $t \ge 0$, implies that $\varphi_* =\lim_{t\to \infty}\varphi(x_t)$ exists. In addition,  one can observe from  
\eqref{exact-ls-bound-3} that \eqref{eq:lem-rec} holds for $\beta_t=\delta^*_{t+\tilde t_0}$, $\gamma_t=\varphi(x_{t+\tilde t_0})-\varphi_*$, $\alpha =1/\nu$, $c=1/4$, and $A=(2M_\nu D_g^{1+\nu})^{\frac{1}{\nu}}$. The inequality \eqref{delta*-uppbnd-3} then directly follows from Lemma~\ref{lem:rec}.

(ii) One can observe from  \eqref{exact-ls-bound-3} that \eqref{eq:lem-rec} holds for $\beta_t=\delta^*_{t+\tilde t_0}$, $\gamma_t=\varphi(x_{t+\tilde t_0})-\varphi^*$, $\alpha =1/\nu$, $c=1/4$, and $A=(2M_\nu D_g^{1+\nu})^{\frac{1}{\nu}}$. In addition, by Lemma~\ref{lem:FW-gap-bound-convex}, and the monotonicity of $\{\varphi(x_t)\}$, one has 
\[
\delta_t^*= \min_{0\leq i\leq t}\delta_i \ge \min_{0\leq i\leq t} \left\{ \varphi(x_i) -\varphi^*\right\} = \varphi(x_t)-\varphi^*, \quad \forall t \ge 0.
\]
Hence, $\beta_t \ge \gamma_t$ for all $t\ge 0$. The conclusion of this statement then follows from Lemma~\ref{lem:rec-conv}\,(ii).
\hfill$\blacksquare$

%

\vskip 0.1in
\noindent
\textbf{Proof of Theorem~\ref{th:main2}}~~
It follows from \eqref{delta_t-lwrbnd} and $\delta_t^*= \min_{0\leq i\leq t}\delta_i$ that 
\begin{equation}\label{eq:pr-main-th-ratio-bound2}
\frac{\delta_i}{\norm{x_i-v_i}^2} = \left(\frac{\delta_i}{\norm{x_i-v_i}^\rho}\right)^{\frac{2}{\rho}}\delta_i^{1-\frac{2}{\rho}}
\geq
\left(\frac{\kappa}{\rho}\right)^{\frac{2}{\rho}}\delta_i^{1-\frac{2}{\rho}} \geq
\left(\frac{\kappa}{\rho}\right)^{\frac{2}{\rho}}(\delta_t^*)^{1-\frac{2}{\rho}}, \quad \forall  0\le i \le t.
\end{equation}
Let $C_t$ be defined by \eqref{C_t}. We next bound $C_t$ from below by considering two cases in view of the definition of $\overline{L}(\cdot)$ in \eqref{eq:def-tilde-L}.

Case 1) $\overline{L}(\delta_t^*)=\left(\frac{2(1-\nu)}{1+\nu}\right)^{\frac{1-\nu}{2\nu}}M_\nu^{\frac{1}{\nu}}
\left(\frac{\rho}{\kappa (\delta_t^*)^{\rho-1}}\right)^{\frac{1-\nu}{\nu\rho}}$.
By this, \eqref{C_t} and \eqref{eq:pr-main-th-ratio-bound2}, we obtain that 
\begin{align*}
C_t &\geq \frac{1}{2} \left(\frac{2(1-\nu)}{1+\nu}\right)^{-\frac{1-\nu}{2\nu}}
M_\nu^{-\frac{1}{\nu}}
\left(\frac{\rho}{\kappa (\delta_t^*)^{\rho-1}}\right)^{-\frac{1-\nu}{\nu\rho}}
\cdot
\left(\frac{\kappa}{\rho}\right)^{\frac{2}{\rho}}(\delta_t^*)^{1-\frac{2}{\rho}}
\\&=
2^{-\frac{1+\nu}{2\nu}}\left(\frac{1-\nu}{1+\nu}\right)^{-\frac{1-\nu}{2\nu}}
\left[
	\frac
	{
		\left(\frac{\kappa}{\rho}\right)^{\frac{1+\nu}{\rho}}
		(\delta_t^*)^{\frac{\rho-(1+\nu)}{\rho}}
	}
	{M_\nu}
\right]^{\frac{1}{\nu}}=:E_t.
\end{align*}

Case 2) $\overline{L}(\delta_t^*) = \left(\frac{1-\nu}{1+\nu}\frac{1}{\delta_t^*}\right)^{\frac{1-\nu}{1+\nu}}{M_\nu^{\frac{2}{1+\nu}}}$.
By this, \eqref{C_t} and \eqref{eq:pr-main-th-ratio-bound2}, one has 
\begin{align*}
C_t &\geq
\frac{1}{2} \left(\frac{1-\nu}{1+\nu}\frac{1}{\delta_t^*}\right)^{-\frac{1-\nu}{1+\nu}}{M_\nu^{-\frac{2}{1+\nu}}}
\cdot
\left(\frac{\kappa}{\rho}\right)^{\frac{2}{\rho}}(\delta_t^*)^{1-\frac{2}{\rho}}
\\&=
\frac{1}{2}\left(\frac{1-\nu}{1+\nu}\right)^{-\frac{1-\nu}{1+\nu}}
\left[
	\frac
	{
		\left(\frac{\kappa}{\rho}\right)^{\frac{1+\nu}{\rho}}
		(\delta_t^*)^{\frac{\rho-(1+\nu)}{\rho}}
	}
	{M_\nu}
\right]^{\frac{2}{1+\nu}}=E_t^{\frac{2\nu}{1+\nu}}.
\end{align*}
Combining these two cases, and using \eqref{eq:def-tilde-L} and \eqref{C_t}, we conclude that $C_t \geq \min\{E_t, E_t^{\frac{2\nu}{1+\nu}}\}$. By this and $2\nu/(1+\nu) \leq 1$, one can observe that
$$
\min\{1,C_t\} \geq \min\{1, E_t\}, \quad \forall t \ge 0.
$$
In view of this, \eqref{eq:alg2-rec}, and the expression of $E_t$, one has 
\begin{align}
\varphi(x_{t+1})
&\leq
\varphi(x_t) - \frac{\delta_t^*}{4}\min\left\{1, 
2^{-\frac{1+\nu}{2\nu}}\left(\frac{1-\nu}{1+\nu}\right)^{-\frac{1-\nu}{2\nu}}
\left[
	\frac
	{
		\left(\frac{\kappa}{\rho}\right)^{\frac{1+\nu}{\rho}}
		(\delta_t^*)^{\frac{\rho-(1+\nu)}{\rho}}
	}
	{M_\nu}
\right]^{\frac{1}{\nu}}
\right\}
\nonumber\\
&\leq
\varphi(x_t) - \frac{\delta_t^*}{4}\min\left\{1, 
\left[
	\frac
	{
		\left(\frac{\kappa}{\rho}\right)^{\frac{1+\nu}{\rho}}
		(\delta_t^*)^{\frac{\rho-(1+\nu)}{\rho}}
	}
	{2M_\nu}
\right]^{\frac{1}{\nu}}
\right\},
\quad \forall t \geq \tilde t_0,
\label{eq:adap-ls-rec-uc}
\end{align}
where the last inequality is due to $2^{\frac{1+\nu}{2\nu}} \leq 2^{\frac{1+1}{2\nu}}$ and $\frac{1-\nu}{1+\nu}\leq 1$.

(i) In view of \eqref{eq:alg2-rec1} and $\varphi(x_t) \geq \varphi^* \in \RR$, the sequence $\{\varphi(x_t)\}$ is non-increasing and $\varphi_* =\lim_{t\to \infty}\varphi(x_t)$ exists. In addition,  one can observe from  \eqref{eq:adap-ls-rec-uc} that \eqref{eq:lem-rec} holds for $\beta_t=\delta^*_{t+\tilde t_0}$, $\gamma_t=\varphi(x_{t+\tilde t_0})-\varphi_*$, 
$\alpha=(\rho-1-\nu)/(\rho\nu)$, $c=1/4$, and $A = \left(\frac{\rho}{\kappa}\right)^{\frac{1+\nu}{\rho\nu}}(2M_\nu)^{\frac{1}{\nu}}$.
The inequality \eqref{delta*-uppbnd-4} then directly follows from Lemma~\ref{lem:rec}.

(ii) One can observe from  \eqref{eq:adap-ls-rec-uc} that \eqref{eq:lem-rec} holds for $\beta_t=\delta^*_{t+\tilde t_0}$, $\gamma_t=\varphi(x_{t+\tilde t_0})-\varphi^*$, 
$\alpha=(\rho-1-\nu)/(\rho\nu)$, $c=1/4$, and $A = \left(\frac{\rho}{\kappa}\right)^{\frac{1+\nu}{\rho\nu}}(2M_\nu)^{\frac{1}{\nu}}$.
The rest of the proof is the same as that of Theorem~\ref{th:main1}\,(ii).
\hfill$\blacksquare$

\section{Concluding Remarks}
\label{sec:conclusion}

In this paper we first analyzed iteration complexity of a parameter-dependent conditional gradient method for solving problem~\eqref{main-prob}, whose step sizes depend explicitly on the problem parameters. We then proposed a novel parameter-free conditional gradient method for solving \eqref{main-prob} without using any prior knowledge of the problem parameters and showed that it enjoys the same order of iteration complexity as the the parameter-dependent conditional gradient method. Preliminary numerical experiments demonstrate the practical superiority of our parameter-free conditional gradient method over the other variants.

It shall be mentioned that our proposed method requires a pre-specified norm and thus it is norm-dependent. In contrast, some existing conditional gradient methods \cite[e.g.,][]{Jaggi11,Lac16,Pena22} are norm-independent. It would be interesting to develop a parameter-free but norm-independent conditional gradient method achieving the same complexity bounds as obtained this paper for solving \eqref{main-prob}. This is left for future research.

\subsection*{Acknowledgements}
The authors would like to thank the anonymous referees for their valuable comments that improved the quality of the paper.
The first author's research is supported by the Grant-in-Aid for Early-Career Scientists (JP21K17711) from Japan Society for the Promotion of Science. The work of the second author was partially supported by NSF Award IIS-2211491.

\vskip 0.2in
\bibliography{cgm}

\end{document}